\documentclass[11pt]{article}

\usepackage{amsmath, amsfonts, amssymb, amsthm, mathtools, enumerate}
\usepackage{graphicx,subcaption,caption}

\usepackage[blocks, affil-it]{authblk}

\usepackage[numbers, square]{natbib}
\usepackage[CJKbookmarks=true,
            bookmarksnumbered=true,
			bookmarksopen=true,
			colorlinks=true,
			citecolor=red,
			linkcolor=blue,
			anchorcolor=red,
			urlcolor=blue]{hyperref}
\usepackage[usenames]{color}

\usepackage[letterpaper, left=1.2truein, right=1.2truein, top = 1.2truein, bottom = 1.2truein]{geometry}
\usepackage[ruled, vlined, lined, commentsnumbered]{algorithm2e}

\usepackage{prettyref,soul}

\newtheorem{lemma}{Lemma}[section]
\newtheorem{proposition}{Proposition}[section]
\newtheorem{thm}{Theorem}[section]

\newtheorem{corollary}{Corollary}[section]

\newrefformat{eq}{(\ref{#1})}
\newrefformat{chap}{Chapter~\ref{#1}}
\newrefformat{sec}{Section~\ref{#1}}
\newrefformat{algo}{Algorithm~\ref{#1}}
\newrefformat{fig}{Fig.~\ref{#1}}
\newrefformat{tab}{Table~\ref{#1}}
\newrefformat{rmk}{Remark~\ref{#1}}
\newrefformat{clm}{Claim~\ref{#1}}
\newrefformat{def}{Definition~\ref{#1}}
\newrefformat{cor}{Corollary~\ref{#1}}
\newrefformat{lmm}{Lemma~\ref{#1}}
\newrefformat{lemma}{Lemma~\ref{#1}}
\newrefformat{prop}{Proposition~\ref{#1}}
\newrefformat{app}{Appendix~\ref{#1}}
\newrefformat{ex}{Example~\ref{#1}}
\newrefformat{exer}{Exercise~\ref{#1}}
\newrefformat{soln}{Solution~\ref{#1}}
\newrefformat{cond}{Condition~\ref{#1}}

\def\text#1{\mbox{\rm #1}}

\DeclarePairedDelimiter{\ceil}{\lceil}{\rceil}

\newcommand{\argmin}{\mathop{\rm argmin}}

\newcommand{\norm}[1]{\|{#1} \|}

\newcommand{\Prob}{\mathbb{P}}
\newcommand{\E}{\mathbb{E}}

\newcommand{\iprod}[2]{\left \langle #1, #2 \right\rangle}

\newcommand{\floor}[1]{{\left\lfloor {#1} \right \rfloor}}

\newcommand{\dimnone}{n_1}
\newcommand{\dimntwo}{n_2}
\newcommand{\dimkone}{k_1}
\newcommand{\dimktwo}{k_2}

\title{Optimal Estimation and Completion of Matrices with Biclustering Structures
}
\author[1]{Chao Gao}
\author[1]{Yu Lu}
\author[2]{Zongming Ma}
\author[1]{Harrison H.~Zhou}
\affil[1]{
Yale University
}
\affil[2]{
University of Pennsylvania
}
\date{ }

\begin{document}
\maketitle

\begin{abstract}
Biclustering structures in data matrices were first formalized in a seminal paper by John Hartigan \cite{hartigan1972direct} where one seeks to cluster cases and variables simultaneously.
Such structures are also prevalent in block modeling of networks.
In this paper, we develop a theory for the estimation and completion of matrices with biclustering structures, where the data is a partially observed and noise contaminated matrix with a certain underlying biclustering structure.
In particular, we show that a constrained least squares estimator achieves minimax rate-optimal performance in several of the most important scenarios.
To this end, we derive unified high probability upper bounds for all sub-Gaussian data and also provide matching minimax lower bounds in both Gaussian and binary cases.
Due to the close connection of graphon to stochastic block models, an immediate consequence of our general results is a minimax rate-optimal estimator for sparse graphons.
\smallskip

\end{abstract}


\section{Introduction}

In a range of important data analytic scenarios, we encounter matrices with biclustering structures.
For instance, in gene expression studies, one can organize the rows of a data matrix to correspond to individual cancer patients and the columns to transcripts. 
Then the patients are expected to form groups according to different cancer subtypes and the genes are also expected to exhibit clustering effect according to the different pathways they belong to.
Therefore, after appropriate reordering of the rows and the columns, the data matrix is expected to have a biclustering structure contaminated by noises \cite{lee2010}.
Here, the observed gene expression levels are real numbers.
In a different context, such a biclustering structure can also be present in network data.
For example, stochastic block model (SBM for short) \cite{holland1983stochastic} is a popular model for exchangeable networks. 
In SBMs, the graph nodes are partitioned into $k$ disjoint communities and the probability that any pair of nodes are connected is determined entirely by the community memberships of the nodes.
Consequently, if one rearranges the nodes from the same communities together in the graph adjacency matrix, then the mean adjacency matrix, where each off-diagonal entry equals the probability of an edge connecting the nodes represented by the corresponding row and column, also has a biclustering structure.

The goal of the present paper is to develop a theory for the estimation (and completion when there are missing entries) of matrices with biclustering structures.
To this end, we propose to consider the following general model
\begin{align}
\label{eq:model}
X_{ij}=\theta_{ij}+\epsilon_{ij}, \quad i\in[n_1], j\in[n_2],
\end{align}
where for any positive integer $m$, we let $[m] = \{1,\dots, m\}$.
Here, for each $(i,j)$, $\theta_{ij} = \mathbb{E}[X_{ij}]$ and $\epsilon_{ij}$ is an independent piece of mean zero sub-Gaussian noise.
Moreover, we allow entries to be missing completely at random \citep{rubin1976inference}.
Thus, let $E_{ij}$ be i.i.d.~Bernoulli random variables with success probability $p\in (0,1]$ indicating whether the $(i,j)$th entry is observed, and define the set of observed entries
\begin{align}
\label{eq:Omega}
\Omega = \{(i,j): E_{ij}=1\}.
\end{align}
Our final observations are
\begin{align}
	\label{eq:observed}
	X_{ij}, \quad (i,j)\in \Omega.
\end{align}
To model the biclustering structure, we focus on the case where 
there are $k_1$ row clusters and $k_2$ column clusters, and the values of $\{\theta_{ij}\}$ are taken as constant if the rows and the columns belong to the same clusters.
The goal is then to recover the signal matrix $\theta \in \mathbb{R}^{n_1\times n_2}$ from the observations \eqref{eq:observed}.
To accommodate most interesting cases, especially the case of undirected networks, we shall also consider the case where the data matrix $X$ is symmetric with zero diagonals. In such cases, we also require $X_{ij} = X_{ji}$ and $E_{ij} = E_{ji}$ for all $i\neq j$.

\paragraph{Main contributions}
In this paper, we propose a unified estimation procedure for partially observed data matrix generated from model \eqref{eq:model} -- \eqref{eq:observed}.
We establish high probability upper bounds for the mean squared errors of the resulting estimators.
In addition, we show that
these upper bounds are minimax rate-optimal in both the continuous case and the binary case by providing matching minimax lower bounds.
Furthermore, SBM can be viewed as a special case of the symmetric version of \eqref{eq:model}. Thus, an immediate application of our results is the network completion problem for SBMs. 
With partially observed network edges, our method gives a rate-optimal estimator for the probability matrix of the whole network in both the dense and the sparse regimes, which further leads to rate-optimal graphon estimation in both regimes.

\paragraph{Connection to the literature}
If only a low rank constraint is imposed on the mean matrix $\theta$, then \eqref{eq:model} -- \eqref{eq:observed} becomes what is known in the literature as the matrix completion problem \cite{recht2010guaranteed}.
An impressive list of algorithms and theories have been developed for this problem, including but not limited to \citep{candes2009exact,keshavan2009matrix,candes2010power,candes2010matrix,cai2010singular,keshavan2010matrix,recht2011simpler,koltchinskii2011nuclear}. 
In this paper, we investigate an alternative biclustering structural assumption for the matrix completion problem, which was first proposed by John Hartigan \cite{hartigan1972direct}.
Note that a biclustering structure automatically implies low-rankness. 
However, if one applies a low rank matrix completion algorithm directly in the current setting,  the resulting estimator suffers an inferior error bound to the minimax rate-optimal one. 
Thus, a full exploitation of the biclustering structure is necessary, which is the focus of the current paper.

The results of our paper also imply rate-optimal estimation for sparse graphons. 
Previous results on graphon estimation include \cite{airoldi2013stochastic,wolfe2013nonparametric,olhede2014network,borgs2015consistent,choi2015co} and the references therein. The minimax rates for dense graphon estimation were derived by \cite{gao2014rate}. 
During the time when this paper was written, we became aware of an independent result on optimal sparse graphon estimation by \cite{klopp2015oracle}. 

There are also an interesting line of works on biclustering \citep{flynn2012consistent,rohe2012co,choi2014co}. While these papers aim to recover the clustering structures of rows and columns, the goal of the current paper is to estimate the underlying mean matrix with optimal rates.

\paragraph{Organization}
After a brief introduction to notation, the rest of the paper is organized as follows. 
In Section \ref{sec:method}, we introduce the precise formulation of the problem and propose a constrained least squares estimator for the mean matrix $\theta$. 
In Section \ref{sec:theory}, we show that the proposed estimator leads to minimax optimal performance for both Gaussian and binary data. 
Section \ref{sec:ext} presents some extensions of our results to sparse graphon estimation and adaptation. 
Implementation and simulation results are given in Section \ref{sec:simulation}.
In Section \ref{sec:discussion}, we discuss the key points of the paper and propose some open problems for future research.
The proofs of the main results are laid out in Section \ref{sec:pf}, with some auxiliary results deferred to the appendix.

\paragraph{Notation}
For a vector $z\in[k]^n$, define the set $z^{-1}(a)=\{i\in[n]:z(i)=a\}$ for $a\in[k]$. For a set $S$, $|S|$ denotes its cardinality and $\mathbf{1}_S$ denotes the indicator function. For a matrix $A=(A_{ij})\in\mathbb{R}^{n_1\times n_2}$, the $\ell_2$ norm and $\ell_{\infty}$ norm are defined by $\norm{A}=\sqrt{\sum_{ij}A_{ij}^2}$ and $\norm{A}_{\infty}=\max_{ij}|A_{ij}|$, respectively. The inner product for two matrices $A$ and $B$ is $\iprod{A}{B}=\sum_{ij}A_{ij}B_{ij}$. Given a subset $\Omega\in[n_1]\times [n_2]$, we use the notation $\iprod{A}{B}_{\Omega}=\sum_{(i,j)\in\Omega}A_{ij}B_{ij}$ and $\norm{A}_{\Omega}=\sqrt{\sum_{(i,j)\in\Omega}A_{ij}^2}$. Given two numbers $a,b\in\mathbb{R}$, we use $a\vee b=\max(a,b)$ and $a\wedge b=\min(a,b)$. The floor function $\floor{a}$ is the largest integer no greater than $a$, and the ceiling function $\ceil{a}$ is the smallest integer no less than $a$. For two positive sequences $\{a_n\},\{b_n\}$, $a_n\lesssim b_n$ means $a_n\leq Cb_n$ for some constant $C>0$ independent of $n$, and $a_n\asymp b_n$ means $a_n\lesssim b_n$ and $b_n\lesssim a_n$.  The symbols $\mathbb{P}$ and $\mathbb{E}$ denote generic probability and expectation operators whose distribution is determined from the context.

\section{Constrained least squares estimation}
\label{sec:method}

Recall the generative model defined in \eqref{eq:model} and also the definition of the set $\Omega$ in \eqref{eq:Omega} of the observed entries.
As we have mentioned, throughout the paper, we assume that the $\epsilon_{ij}$'s are independent sub-Gaussian noises with sub-Gaussianity parameter uniformly bounded from above by $\sigma > 0$.
More precisely, we assume
\begin{equation}
\mathbb{E}e^{\lambda\epsilon_{ij}}\leq e^{\lambda^2\sigma^2/2},
\quad
\text{for all }i\in [n_1], j\in [n_2]\,\, \mbox{and}\,\, \lambda\in\mathbb{R}.
\label{eq:sub-G}
\end{equation}

We consider two types of biclustering structures. 
One is rectangular and asymmetric, where we assume that the mean matrix 
belongs to the following parameter space
\begin{equation}
\label{eq:para-space-asym}
\begin{aligned}
\Theta_{k_1k_2}(M) = \Big\{\theta = (\theta_{ij})\in\mathbb{R}^{n_1\times n_2} & : \theta_{ij}=Q_{z_1(i)z_2(j)}, z_1\in[k_1]^{n_1},z_2\in[k_2]^{n_2},\\
& ~~~ Q\in[-M,M]^{k_1\times k_2}\Big\}.	
\end{aligned}
\end{equation}
In other words, the mean values within each bicluster is homogenous, i.e., $\theta_{ij}=Q_{ab}$ if the $i$th row belongs to the $a$th row cluster and the $j$th column belong to the $b$th column cluster.
The other type of structures we consider is the square and symmetric case.
In this case, we impose symmetry requirement on the data generating process, i.e., $n_1= n_2 =n$ and
\begin{align}
\label{eq:sym-con}
X_{ij} = X_{ji},\,\, E_{ij} = E_{ji}, \,\,\mbox{for all $i\neq j$.}
\end{align}
Since the case is mainly motivated by undirected network data where there is no edge linking any node to itself, we also assume $X_{ii} = 0$ for all $i\in [n]$.
Finally, the mean matrix is assumed to belong to the following parameter space
\begin{equation}
	\label{eq:para-space-sym}
\begin{aligned}
\Theta^s_k(M)=\Big\{\theta=(\theta_{ij})\in\mathbb{R}^{n\times n} & :
\theta_{ii}=0,\, \theta_{ij}=\theta_{ji}=Q_{z(i)z(j)}\text{ for $i>j$}, z\in[k]^{n},\\
& ~~~ Q = Q^T\in[-M,M]^{k\times k}
\Big\}.
\end{aligned}	
\end{equation}

We proceed by assuming that we know the parameter space $\Theta$ which can be either $\Theta_{k_1 k_2}(M)$ or $\Theta^s_k(M)$ and the rate $p$ of an independent entry being observed, 
The issues of adaptation to unknown numbers of clusters and unknown observation rate $p$ are addressed later in \prettyref{sec:adapt} and \prettyref{sec:adapt-p}.
Given $\Theta$ and $p$, we propose to estimate $\theta$ by the following program 
\begin{equation}
\min_{\theta\in\Theta}\left\{\norm{\theta}^2-\frac{2}{p}\iprod{X}{\theta}_{\Omega}\right\}.
\label{eq:pre-form}
\end{equation}
If we define 
\begin{align}
	\label{eq:data-Y}
Y_{ij}=X_{ij}E_{ij}/p,	
\end{align}
then \eqref{eq:pre-form} is equivalent to the following constrained least squares problem
\begin{equation}
\min_{\theta\in\Theta}\norm{Y-\theta}^2, \label{eq:LS}
\end{equation}
and hence the name of our estimator.
When the data is binary, $\Theta = \Theta^s_k(1)$ and $p = 1$, the problem specializes to estimating the mean adjacency matrix in stochastic block models, and
the estimator defined as the solution to \eqref{eq:LS} reduces to the least squares estimator in \cite{gao2014rate}. 

%
%
%
%
%

\section{Main results}\label{sec:theory}

In this section, we provide theoretical justifications of the constrained least squares estimator defined as the solution to \eqref{eq:LS}.
Our first result is the following universal high probability upper bounds.

\begin{thm}\label{thm:main}
For any global optimizer of (\ref{eq:LS}) and any constant $C'>0$, there exists a constant $C>0$ only depending on $C'$ such that
$$\norm{\hat{\theta}-\theta}^2\leq C\frac{M^2\vee\sigma^2}{p}\left(k_1k_2+n_1\log k_1+n_2\log k_2\right),$$
with probability at least $1-\exp\left(-C'\left(k_1k_2+n_1\log k_1+n_2\log k_2\right)\right)$ uniformly over $\theta\in\Theta_{k_1k_2}(M)$ and all error distributions satisfying (\ref{eq:sub-G}). For the symmetric parameter space $\Theta^s_k(M)$, the bound is simplified to
$$\norm{\hat{\theta}-\theta}^2\leq C\frac{M^2\vee\sigma^2}{p}\left(k^2+n\log k\right),$$
with probability at least $1-\exp\left(-C'\left(k^2+n\log k\right)\right)$ uniformly over $\theta\in\Theta^s_{k}(M)$ and all error distributions satisfying (\ref{eq:sub-G}).
\end{thm}

When $(M^2\vee\sigma^2)$ is bounded, the rate in Theorem \ref{thm:main} is $\left(k_1k_2+n_1\log k_1+n_2\log k_2\right)/p$ which can be decomposed into two parts. 
The part involving $k_1k_2$ reflects the number of parameters in the biclustering structure, 
while the part involving $(n_1\log k_1+n_2\log k_2)$ results from the complexity of estimating the clustering structures of rows and columns. 
It is the price one needs to pay for not knowing the clustering information. 
In contrast, the minimax rate for matrix completion under low rank assumption would be $(n_1\vee n_2)(k_1\wedge k_2)/p$ \cite{koltchinskii2011nuclear,ma2013volume}, 
since without any other constraint the biclustering assumption implies that the rank of the mean matrix $\theta$ is at most $k_1\wedge k_2$.
Therefore, we have $\left(k_1k_2+n_1\log k_1+n_2\log k_2\right)/p\ll (n_1\vee n_2)(k_1\wedge k_2)/p$ as long as both $n_1\vee n_2$ and $k_1\wedge k_2$ tend to infinity. 
Thus, by fully exploiting the biclustering structure, we obtain a better convergence rate than only using the low rank assumption.

In the rest of this section, we discuss two most representative cases, namely the Gaussian case and the symmetric Bernoulli case.
The latter case is also known in the literature as stochastic block models.

\paragraph{The Gaussian case}
Specializing Theorem \ref{thm:main} to Gaussian random variables, we obtain the following result.
\begin{corollary}\label{cor:Gaussian}
Assume $\epsilon_{ij}\stackrel{iid}{\sim} N(0,\sigma^2)$ and $M\leq C_1\sigma$ for some constant $C_1>0$. For any constant $C'>0$, there exists some constant $C$ only depending on $C_1$ and $C'$ such that
$$\norm{\hat{\theta}-\theta}^2\leq C\frac{\sigma^2}{p}\left(k_1k_2+n_1\log k_1+n_2\log k_2\right),$$
with probability at least $1-\exp\left(-C'\left(k_1k_2+n_1\log k_1+n_2\log k_2\right)\right)$ uniformly over $\theta\in\Theta_{k_1k_2}(M)$. For the symmetric parameter space $\Theta^s_k(M)$, the bound is simplified to
$$\norm{\hat{\theta}-\theta}^2\leq C\frac{\sigma^2}{p}\left(k^2+n\log k\right),$$
with probability at least $1-\exp\left(-C'\left(k^2+n\log k\right)\right)$ uniformly over $\theta\in\Theta^s_{k}(M)$.
\end{corollary}

We now present a rate matching lower bound in the Gaussian model to show that the result of Corollary \ref{cor:Gaussian} is minimax optimal. 
To this end, we use $\mathbb{P}_{(\theta,\sigma^2,p)}$ to indicate the probability distribution of the model $X_{ij}\stackrel{ind}{\sim} N(\theta_{ij},\sigma^2)$ with observation rate $p$.
\begin{thm}\label{thm:lower-Gaussian}
Assume $\frac{\sigma^2}{p}\left(\frac{k_1k_2}{n_1n_2}+\frac{\log k_1}{n_2}+\frac{\log k_2}{n_1}\right)\lesssim M^2$.
There exist some constants $C,c>0$, such that
$$\inf_{\hat{\theta}}\sup_{\theta\in\Theta_{k_1k_2}(M)}\mathbb{P}_{(\theta,\sigma^2,p)}\left(\norm{\hat{\theta}-\theta}^2> C\frac{\sigma^2}{p}\left(k_1k_2+n_1\log k_1+n_2\log k_2\right)\right)>c,$$
when $\log k_1\asymp\log k_2$, and
$$\inf_{\hat{\theta}}\sup_{\theta\in\Theta^s_{k}(M)}\mathbb{P}_{(\theta,\sigma^2,p)}\left(\norm{\hat{\theta}-\theta}^2> C\frac{\sigma^2}{p}\left(k^2+n\log k\right)\right)>c.$$
\end{thm}

\paragraph{The symmetric Bernoulli case}
When the observed matrix is symmetric with zero diagonal and Bernoulli random variables as its super-diagonal entries, it can be viewed as the adjacency matrix of an undirected network and the problem of estimating its mean matrix with missing data can be viewed as a network completion problem.
Given a partially observed Bernoulli adjacency matrix $\{X_{ij}\}_{(i,j)\in\Omega}$, one can predict the unobserved edges by estimating the whole mean matrix $\theta$.  Also, we assume that each edge is observed independently with probability $p$.

Given a symmetric adjacency matrix $X=X^T\in\{0,1\}^{n\times n}$ 
with zero diagonals,
the stochastic block model \citep{holland1983stochastic} assumes $\{X_{ij}\}_{i>j}$ are independent Bernoulli random variables with mean $\theta_{ij}=Q_{z(i)z(j)}\in[0,1]$ with some matrix $Q\in[0,1]^{k\times k}$ and some label vector $z\in [k]^n$. 
In other words, the probability that there is an edge between the $i$th and the $j$th nodes only depends on their community labels $z(i)$ and $z(j)$. 
The following class then includes all possible mean matrices of stochastic block models with $n$ nodes and $k$ clusters and with edge probabilities uniformly bounded by $\rho$: 
\begin{align}
\Theta_k^+(\rho)=\left\{\theta\in[0,1]^{n\times n}: \theta_{ii}=0, \theta_{ij}=\theta_{ji}=Q_{z(i)z(j)}, Q=Q^T\in[0,\rho]^{k\times k},z\in[k]^n\right\}.	
\end{align}
By the definition in \eqref{eq:para-space-sym}, $\Theta^+_k(\rho)\subset \Theta^s_k(\rho)$.

Although the tail probability of Bernoulli random variables does not satisfy the sub-Gaussian assumption (\ref{eq:sub-G}),  a slightly modification of the proof of Theorem \ref{thm:main} leads to the following result. The proof of Corollary \ref{cor:SBM} will be given in Section \ref{sec:auxiliary} in the appendix.
\begin{corollary}
	\label{cor:SBM}
Consider the optimization problem (\ref{eq:LS}) with $\Theta = \Theta^s_k(\rho)$. For any global optimizer $\hat{\theta}$ and any constant $C'>0$, there exists a constant $C>0$ only depending on $C'$ such that
$$\norm{\hat{\theta}-\theta}^2\leq C\frac{\rho}{p}\left(k^2+n\log k\right),$$
with probability at least $1-\exp\left(-C'\left(k^2+n\log k\right)\right)$ uniformly over $\theta\in \Theta^s_k(\rho)\supset\Theta^+_{k}(\rho)$.
\end{corollary}

When $\rho=p=1$, Corollary \ref{cor:SBM} implies Theorem 2.1 in \cite{gao2014rate}. 
A rate matching lower bound is given by the following theorem. We denote the probability distribution of a stochastic block model with mean matrix $\theta\in\Theta^+_k(\rho)$ and observation rate $p$ by $\mathbb{P}_{(\theta,p)}$.
\begin{thm}\label{thm:lower-SBM}
For stochastic block models, we have
$$\inf_{\hat{\theta}}\sup_{\theta\in\Theta^+_k(\rho)}\mathbb{P}_{(\theta,p)}\left(\norm{\hat{\theta}-\theta}^2>C\left(\frac{\rho\left(k^2+n\log k\right)}{p}\wedge \rho^2n^2\right)\right)>c,$$
for some constants $C,c>0$.
\end{thm}
The lower bound is the minimum of two terms. When $\rho\geq \frac{k^2+n\log k}{pn^2}$, the rate becomes $\frac{\rho(k^2+n\log k)}{p}\wedge \rho^2n^2\asymp\frac{\rho(k^2+n\log k)}{p}$. 
It is achieved by the constrained least squares estimator according to Corollary \ref{cor:SBM}. When $\rho< \frac{k^2+n\log k}{pn^2}$, the rate is dominated by $\rho^2n^2$. In this case, a trivial zero estimator achieves the minimax rate.

In the case of $p=1$, a comparable result has been found independently by \cite{klopp2015oracle}. 
However, our result here is more general as it accommodates missing observations. 
Moreover, the general upper bounds in \prettyref{thm:main} even hold for networks with weighted edges.

\section{Extensions}\label{sec:ext}

In this section, we extends the estimation procedure and the theory in Sections \ref{sec:method} and \ref{sec:theory} toward three directions: adaptation to unknown observation rate, adaptation to unknown model parameters,  and sparse graphon estimation.

\subsection{Adaptation to unknown observation rate}
\label{sec:adapt-p}

The estimator \eqref{eq:LS} depends on the knowledge of the observation rate $p$. 
When $p$ is not too small, such a knowledge is not necessary for achieving the desired rates.
Define 
\begin{align}
\label{eq:phat-asym}
\hat{p} = \frac{\sum_{i=1}^{n_1}\sum_{j=1}^{n_2} E_{ij}}{n_1n_2}
\end{align}
for the asymmetric and 
\begin{align}
\label{eq:phat-sym}
\hat{p} = \frac{\sum_{1\leq i<j\leq n} E_{ij}}{\frac{1}{2}n(n-1)}
\end{align}
for the symmetric case, and redefine
\begin{align}
\label{eq:Y-adapt}
Y_{ij} = X_{ij}E_{ij}/\hat{p}
\end{align}
where the actual definition of $\hat{p}$ is chosen between \eqref{eq:phat-asym} and \eqref{eq:phat-sym} depending on whether one is dealing with the asymmetric or symmetric parameter space.
Then we have the following result for the solution to \eqref{eq:LS} with $Y$ redefined by \eqref{eq:Y-adapt}.


\begin{thm}\label{thm:p}
For $\Theta = \Theta_{k_1k_2}(M)$, suppose for some absolute constant $C_1 > 0$,
\begin{align*}
p\geq C_1 \frac{[\log(n_1+n_2)]^2}{k_1k_2+n_1\log k_1+n_2\log k_2}.	
\end{align*}
Let $\hat\theta$ be the solution to \eqref{eq:LS} with $Y$ defined as in \eqref{eq:Y-adapt}. 
Then for any constant $C'>0$, there exists a constant $C>0$ only depending on $C'$ and $C_1$ such that
$$\norm{\hat{\theta}-\theta}^2\leq C\frac{M^2\vee\sigma^2}{p}\left(k_1k_2+n_1\log k_1+n_2\log k_2\right),$$
with probability at least $1-(n_1n_2)^{-C'}$ uniformly over $\theta\in\Theta$ and all error distributions satisfying (\ref{eq:sub-G}).

For $\Theta=\Theta^s_k(M)$, the same result holds if we replace $n_1$ and $n_2$ with $n$ and $k_1$ and $k_2$ with $k$ in the foregoing statement.
\end{thm}

\subsection{Adaptation to unknown model parameters}
\label{sec:adapt}

We now provide an adaptive procedure for estimating $\theta$ without assuming the knowledge of the model parameters $k_1$, $k_2$ and $M$. 
The procedure can be regarded as a variation of a 2-fold cross validation \citep{wold1978cross}.
We give details on the procedure for the asymmetric parameter spaces $\Theta_{k_1k_2}(M)$, and that for the symmetric parameter spaces $\Theta^s_{k}(M)$ can be obtained similarly.

To adapt to $k_1$, $k_2$ and $M$, we split the data into two halves. 
Namely, sample i.i.d. $T_{ij}$ from Bernoulli$(\frac{1}{2})$. 
Define $\Delta=\{(i,j)\in[n_1]\times n_2:T_{ij}=1\}$. 
Define $Y_{ij}^{\Delta}=2X_{ij}E_{ij}T_{ij}/p$ and $Y_{ij}^{\Delta^c}=2X_{ij}E_{ij}(1-T_{ij})/p$ for all $(i,j)\in[n_1]\times [n_2]$. 
Then, for some given $(k_1,k_2,M)$, the least squares estimators using $Y^{\Delta}$ and $Y^{\Delta^c}$ are given by
$$\hat{\theta}^{\Delta}_{k_1k_2M}=\argmin_{\theta\in\Theta_{k_1k_2}(M)}\norm{Y^{\Delta}-\theta}^2,\quad \hat{\theta}^{\Delta^c}_{k_1k_2M}=\argmin_{\theta\in\Theta_{k_1k_2}(M)}\norm{Y^{\Delta^c}-\theta}^2.$$
Select the parameters by
$$(\hat{k}_1,\hat{k}_2,\hat{M})=\argmin_{(k_1,k_2,M)\in[n_1]\times[n_2]\times\mathcal{M}}\norm{\hat{\theta}_{k_1k_2M}^{\Delta}-Y^{\Delta^c}}^2_{\Delta^c},$$
where $\mathcal{M}=\left\{\frac{h}{n_1+n_2}: h\in[(n_1+n_2)^6]\right\}$,
and define $\hat{\theta}^{\Delta}=\hat{\theta}^{\Delta}_{\hat{k}_1\hat{k}_2\hat{M}}$. Similarly, we can also define $\hat{\theta}^{\Delta^c}$ by validate the parameters using $Y^{\Delta}$. The final estimator is given by
$$\hat{\theta}_{ij}=\begin{cases}
\hat{\theta}_{ij}^{\Delta^c}, & (i,j)\in\Delta; \\
\hat{\theta}_{ij}^{\Delta}, & (i,j)\in\Delta^c.
\end{cases}$$

\begin{thm}\label{thm:adaptive}
Assume $(n_1+n_2)^{-1}\leq M\leq (n_1+n_2)^5-(n_1+n_2)^{-1}$.
For any constant $C'>0$, there exists a constant $C>0$ only depending on $C'$ such that
$$\norm{\hat{\theta}-\theta}^2\leq C\frac{M^2\vee\sigma^2}{p}\left(k_1k_2+n_1\log k_1+n_2\log k_2+\frac{\log(n_1+n_2)}{p}\right),$$
with probability at least $1-\exp\left(-C'\left(k_1k_2+n_1\log k_1+n_2\log k_2\right)\right)-(n_1n_2)^{-C'}$ uniformly over $\theta\in\Theta_{k_1k_2}(M)$ and all error distributions satisfying (\ref{eq:sub-G}).
\end{thm}
Compared with Theorem \ref{thm:main}, the rate given by Theorem \ref{thm:adaptive} has an extra $p^{-1}\log(n_1+n_2)$ term. A sufficient condition for this extra term to be inconsequential is $p\gtrsim \frac{\log(n_1+n_2)}{n_1\wedge n_2}$.
Theorem \ref{thm:adaptive} is adaptive for all $(k_1,k_2)\in[n_1]\times[n_2]$ and for $(n_1+n_2)^{-1}\leq M\leq (n_1+n_2)^5-(n_1+n_2)^{-1}$. In fact, by choosing a larger $\mathcal{M}$, we can extend the adaptive region for $M$ to $(n_1+n_2)^{-a}\leq M\leq (n_1+n_2)^b$ for arbitrary constants $a,b>0$.

\subsection{Sparse graphon estimation}

Consider a random graph with adjacency matrix $\{X_{ij}\}\in\{0,1\}^{n\times n}$, whose sampling procedure is determined by
\begin{equation}
(\xi_1,...,\xi_n)\sim\mathbb{P}_{\xi},\quad X_{ij}|(\xi_i,\xi_j)\sim \text{Bernoulli}(\theta_{ij}),\quad\text{where }\theta_{ij}=f(\xi_i,\xi_j).\label{eq:graphon-gen}
\end{equation}
For $i\in[n]$, $X_{ii}=\theta_{ii}=0$. Conditioning on $(\xi_1,...,\xi_n)$, $X_{ij}=X_{ji}$ is independent across $i>j$. The function $f$ on $[0,1]^2$, which is assumed to be symmetric, is called a graphon. 
The concept of graphon is originated from graph limit theory \citep{hoover79,lovasz06,diaconis07,lovasz12} and the studies of exchangeable arrays \citep{aldous81,kallenberg89}. It is the underlying nonparametric object that generates the random graph. Statistical estimation of graphon has been considered by \cite{wolfe2013nonparametric,olhede2014network,gao2014rate,gao2015,lu2015} for dense networks. Using Corollary \ref{cor:SBM}, we present a result for sparse graphon estimation.

Let us start with specifying the function class of graphons. Define the derivative operator by
$$\nabla_{jk}f(x,y)=\frac{\partial^{j+k}}{(\partial x)^j(\partial y)^k}f(x,y),$$
and we adopt the convention $\nabla_{00}f(x,y)=f(x,y)$.
The H\"{o}lder norm is defined as
$$||f||_{\mathcal{H}_{\alpha}}=\max_{j+k\leq\floor{\alpha}}\sup_{x,y\in\mathcal{D}}\left|\nabla_{jk}f(x,y)\right|+\max_{j+k=\floor{\alpha}}\sup_{(x,y)\neq (x',y')\in\mathcal{D}}\frac{\left|\nabla_{jk}f(x,y)-\nabla_{jk}f(x',y')\right|}{||(x-x',y-y')||^{\alpha-\floor{\alpha}}},$$
where $\mathcal{D}=\{(x,y)\in[0,1]^2:x\geq y\}$. Then, the sparse graphon class with H\"{o}lder smoothness $\alpha$ is defined by
$$\mathcal{F}_{\alpha}(\rho,L)=\left\{0\leq f\leq \rho: \norm{f}_{\mathcal{H}_{\alpha}}\leq L\sqrt{\rho}, f(x,y)=f(y,x)\text{ for all }x\in\mathcal{D}\right\},$$
where $L>0$ is the radius of the class, which is assumed to be a constant.
As argued in \cite{gao2014rate}, it is sufficient to approximate a graphon with H\"{o}lder smoothness by a piecewise constant function. In the random graph setting, a piecewise constant function is the stochastic block model. Therefore, we can use the estimator defined by (\ref{eq:LS}). Using Corollary \ref{cor:SBM}, a direct bias-variance tradeoff argument leads to the following result.
An independent finding of the same result is also made by 
\cite{klopp2015oracle}.

\begin{corollary}\label{cor:spg}
Consider the optimization problem (\ref{eq:LS}) where $Y_{ij}=X_{ij}$ and $\Theta = \Theta^s_k(M)$ with $k =\ceil{n^{\frac{1}{1+\alpha\wedge 1}}}$ and $M=\rho$. 
Given any global optimizer $\hat{\theta}$ of \eqref{eq:LS}, we estimate $f$ by $\hat{f}(\xi_i,\xi_j)=\hat{\theta}_{ij}$. Then, for any constant $C'>0$, there exists a constant $C>0$ only depending on $C'$ and $L$ such that
$$\frac{1}{n^2}\sum_{i,j\in[n]}\left(\hat{f}(\xi_i,\xi_j)-f(\xi_i,\xi_j)\right)^2\leq C\rho\left(n^{-\frac{2\alpha}{\alpha+1}}+\frac{\log n}{n}\right),$$
with probability at least $1-\exp(-C'(n^{\frac{1}{\alpha+1}}+n\log n) )$ uniformly over $f\in\mathcal{F}_{\alpha}(\rho,L)$ and $\mathbb{P}_{\xi}$.
\end{corollary}
 Corollary \ref{cor:spg} implies an interesting phase transition phenomenon.
When $\alpha\in(0,1)$, the rate becomes $\rho (n^{-\frac{2\alpha}{\alpha+1}}+\frac{\log n}{n} )\asymp \rho n^{-\frac{2\alpha}{\alpha+1}}$, which is the typical nonparametric rate times a sparsity index of the network. When $\alpha\geq 1$, the rate becomes $\rho (n^{-\frac{2\alpha}{\alpha+1}}+\frac{\log n}{n} )\asymp \frac{\rho\log n}{n}$, which does not depend on the smoothness $\alpha$. Corollary \ref{cor:spg} extends Theorem 2.3 of \cite{gao2014rate} to the case $\rho<1$. In \cite{wolfe2013nonparametric}, the graphon $f$ is defined in a different way. Namely, they considered the setting where $(\xi_1,...,\xi_n)$ are i.i.d. Unif$[0,1]$ random variables under $\mathbb{P}_{\xi}$. Then, the adjacency matrix is generated with Bernoulli random variables having means $\theta_{ij}=\rho f(\xi_i,\xi_j)$ for a nonparametric graphon $f$ satisfying $\int_0^1\int_0^1 f(x,y)dxdy=1$. For this setting, with appropriate smoothness assumption, we can estimate $f$ by $\hat{f}(\xi_i,\xi_j)=\hat{\theta}_{ij}/\rho$. The rate of convergence would be $\rho^{-1} (n^{-\frac{2\alpha}{\alpha+1}}+\frac{\log n}{n} )$.

Using the result of Theorem \ref{thm:adaptive}, we present an adaptive version for Corollary \ref{cor:spg}. The estimator we consider is a symmetric version of the one introduced in Section \ref{sec:adapt}. The only difference is that we choose the set $\mathcal{M}$ as $\{m/n: m\in[n+1]\}$. The estimator is fully data driven in the sense that it does not depend on $\alpha$ or $\rho$.
\begin{corollary}\label{cor:spg-adaptive}
Assume $\rho\geq n^{-1}$.
Consider the adaptive estimator $\hat{\theta}$ introduced in Theorem \ref{thm:adaptive}, and we set $\hat{f}(\xi_i,\xi_j)=\hat{\theta}_{ij}$. Then, for any constant $C'>0$, there exists a constant $C>0$ only depending on $C'$ and $L$ such that
$$\frac{1}{n^2}\sum_{i,j\in[n]}\left(\hat{f}(\xi_i,\xi_j)-f(\xi_i,\xi_j)\right)^2\leq C\rho\left(n^{-\frac{2\alpha}{\alpha+1}}+\frac{\log n}{n}\right),$$
with probability at least $1-n^{-C'}$ uniformly over $f\in\mathcal{F}_{\alpha}(\rho,L)$ and $\mathbb{P}_{\xi}$.
\end{corollary}

\section{Numerical Studies}\label{sec:simulation}

To introduce an algorithm solving (\ref{eq:pre-form}) or (\ref{eq:LS}), we write (\ref{eq:LS}) in an alternative way,
$$\min_{z_1\in[k_1]^{n_1},z_2\in[k_2]^{n_2},Q\in[l,u]^{k_1\times k_2}}L(Q,z_1,z_2),$$
where $l$ and $u$ are the lower and upper constraints of the parameters and
$$L(Q,z_1,z_2)=\sum_{(i,j)\in[n_1]\times[n_2]}(Y_{ij}-Q_{z_1(i)z_2(j)})^2.$$
For biclustering, we set $l=-M$ and $u=M$. For SBM, we set $l=0$ and $u=\rho$. We do not impose symmetry for SBM to gain computational convenience without losing much statistical accuracy.
The simple form of $L(Q,z_1,z_2)$ indicates that we can iteratively optimize over $(Q,z_1,z_2)$ with explicit formulas. This motivates the following algorithm.

\begin{algorithm}
\DontPrintSemicolon
\SetKwInOut{Input}{Input}\SetKwInOut{Output}{Output}
\Input{$\{X_{ij}\}_{(i,j)\in\Omega}$, the numbers of column and row clusters $(k_1, k_2)$, lower and upper constraints $(l,u)$ and the number of random starting points $m$. }
\Output{An $n_1 \times n_2$ matrix $\hat{\theta}$ with $\hat{\theta}_{ij} = Q_{z_1(i) z_2(j)}$.} 
\textbf{Preprocessing:} Let $X_{ij} = 0$ for $(i,j) \not\in \Omega$, $\hat{p}=|\Omega|/(n_1n_2)$ and $Y = X/\hat{p}$. \\
\nl \textbf{Initialization Step} \\
~~~~~Apply singular value decomposition and obtain $Y=UDV^T.$\;
~~~~~Run $k$-means algorithm on the first $k_1$ columns of $U$ with $m$ random starting points to get $z_1$.\;
~~~~~Run $k$-means algorithm on the first $k_2$ columns of $V$ with $m$ random starting points to get $z_2$.\;

    \While{not converge}{\label{InRes1}
	\nl Update $Q$: for each $(a,b) \in [k_1] \times [k_2], $ \; 
	\begin{equation}
	Q_{ab} = \frac{1}{|z_1^{-1}(a)||z_2^{-1}(b)|} \sum_{i \in z_1^{-1}(a)} \sum_{j \in z_2^{-1}(b)} Y_{ij}. \label{eq:laoluan1}
	\end{equation}
	If $Q_{ab} > u$, let $Q_{ab} = u$. If $Q_{ab}<l$, let $Q_{ab}=l$. \;
	\nl Update $z_1$: for each $i \in [n_1]$,
	\begin{equation} 
	z_1(i) = \argmin_{a \in [k_1]} \sum_{j=1}^{n_2} (Q_{az_2(j)} - A_{ij})^2.\label{eq:laoluan2}
	\end{equation}
	\nl Update $z_2$: for each $j \in [n_2]$,	
		\begin{equation} 
		z_2(j) = \argmin_{b \in [k_2]} \sum_{i=1}^{n_1} (Q_{z_1(i)b} - A_{ij})^2.\label{eq:laoluan3}
		\end{equation}
    }
\caption{A Biclustering Algorithm\label{alg:bl}}
\end{algorithm}

The iteration steps (\ref{eq:laoluan1}), (\ref{eq:laoluan2}) and (\ref{eq:laoluan3}) of Algorithm \ref{alg:bl} can be equivalently expressed as
$$Q=\argmin_{Q\in[l,u]^{k_1\times k_2}}L(Q,z_1,z_2);$$
$$z_1=\argmin_{z_1\in[k_1]^{n_1}}L(Q,z_1,z_2);$$
$$z_2=\argmin_{z_2\in[k_2]^{n_2}}L(Q,z_1,z_2).$$
Thus, each iteration will reduce the value of the objective function. It is worth noting that Algorithm \ref{alg:bl} can be viewed as a two-way extension for the ordinary $k$-means algorithm. Since the objective function is non-convex, one cannot guarantee convergence to global optimum.
We initialize the algorithm via a spectral clustering step using multiple random starting points, which worked well on simulated datasets we present below.

Now we present some numerical results to demonstrate the accuracy of the error rate behavior suggested by Theorem \ref{thm:main} on simulated data. 

\paragraph{Bernoulli case.} Our theoretical result indicates the rate of recovery is $\sqrt{\frac{\rho}{p}\left(\frac{k^2}{n^2}+\frac{\log k}{n}\right)}$ for the root mean squared error (RMSE) $\frac{1}{n}\|\hat{\theta} - \theta\|$. 
When $k$ is not too large, the dominating term is $\sqrt{\frac{\rho\log k}{pn}}$. We are going to confirm this rate by simulation.
We first generate our data from SBM with the number of blocks $k \in \{2,4,8,16\}$. The observation rate $p=0.5$. For every fixed $k$, we use four different $Q=0.5 \mathbf{1}_k \mathbf{1}_k^T + 0.1 t \mathbf{I}_k$ with $t=1,2,3,4$ and generate the community labels $z$ uniformly on $[k]$. Then we calculate the  error $\frac{1}{n}\|\hat{\theta} - \theta\|$.  
Panel (a) of Figure \ref{fig:bernoulli} shows the error versus the sample size $n$. 
In Panel (b), we rescale the x-axis to $N=\sqrt{\frac{pn}{\log k}}$. The curves for different $k$ align well with each other and the error decreases at the rate of $1/N$. This confirms our theoretical results in Theorem \ref{thm:main}.
\begin{figure}[h]
\centering
\begin{subfigure}{.5\textwidth}
  \centering
  \includegraphics[width=\linewidth]{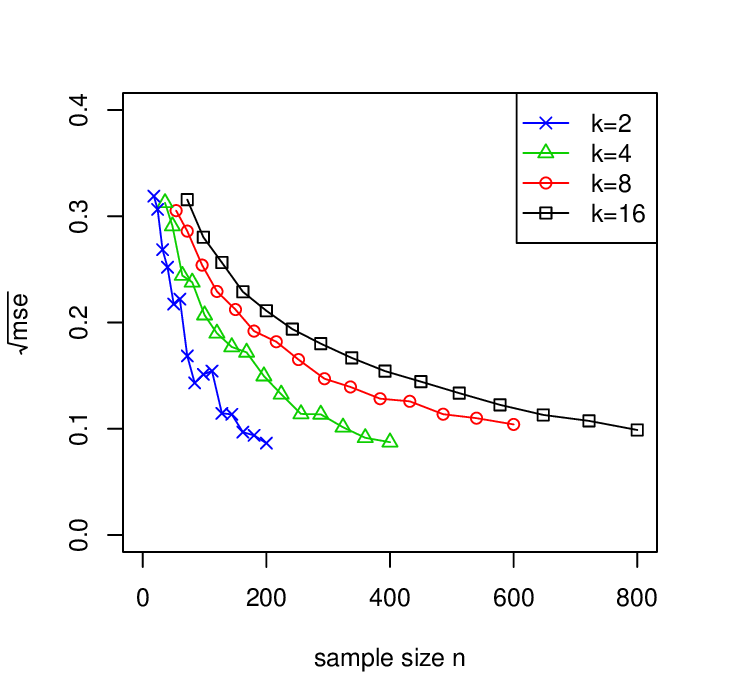}
  \caption{}
  \label{fig:sub1}
\end{subfigure}%
\begin{subfigure}{.5\textwidth}
  \centering
  \includegraphics[width=\linewidth]{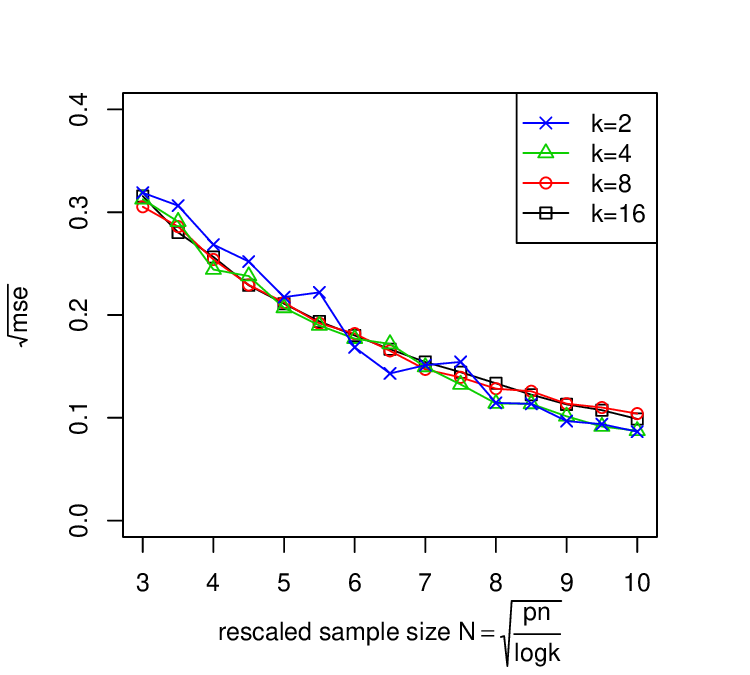}
  \caption{}
  \label{fig:sub2}
\end{subfigure}
\caption{\noindent Plots of $\frac{1}{n}\|\hat{\theta} - \theta\|$  when using our algorithm on SBM. Each curve corresponds to a different sample size $n$ with a fixed $k$. (a) Plots of error against the raw sample size $n$. (b) Plots of the same error against rescaled sample size $\sqrt{pn/\log k}$. 
}
\label{fig:bernoulli}
\end{figure}

\paragraph{Gaussian case.} We simulate data with Gaussian noise under four different settings of $k_1$ and $k_2$. For each $(k_1, k_2) \in \{(4,4), (4,8), (8,8), (8,12) \}$, the entries of matrix $Q$ are independently and uniformly generated from $\{1,2,3,4,5\}$. The cluster labels $z_1$ and $z_2$ are uniform on $[k_1]$ and $[k_2]$ respectively. After generating $Q$, $z_1$ and $z_2$, we add an ${N}(0,1)$ noise to the data and observe $X_{ij}$ with probability $p=0.1$. For each number of rows $n_1$, we set the number of columns as $n_2 = n_1 \log k_1/\log k_2 $. Panel (a) of Figure \ref{fig:gaussian} shows the  error versus  $n_1$. In Panel (b), we rescale the x-axis by $N=\sqrt{\frac{pn_1}{\log k_2}}$. Again, the plots for different $(k_1, k_2)$ align fairly well and the error decreases roughly at the rate of $1/N$.
\begin{figure}[h]
\centering
\begin{subfigure}{.5\textwidth}
  \centering
  \includegraphics[width=\linewidth]{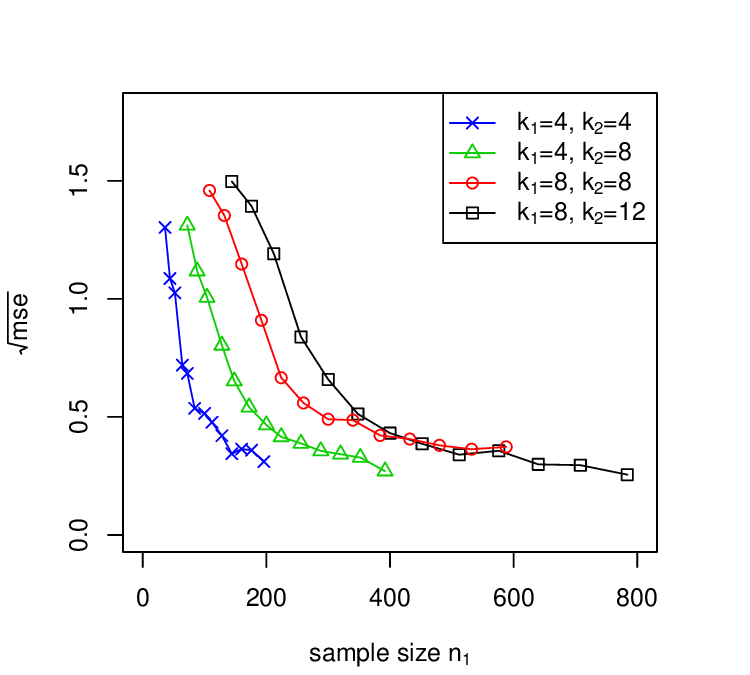}
  \caption{}
\end{subfigure}%
\begin{subfigure}{.5\textwidth}
  \centering
  \includegraphics[width=\linewidth]{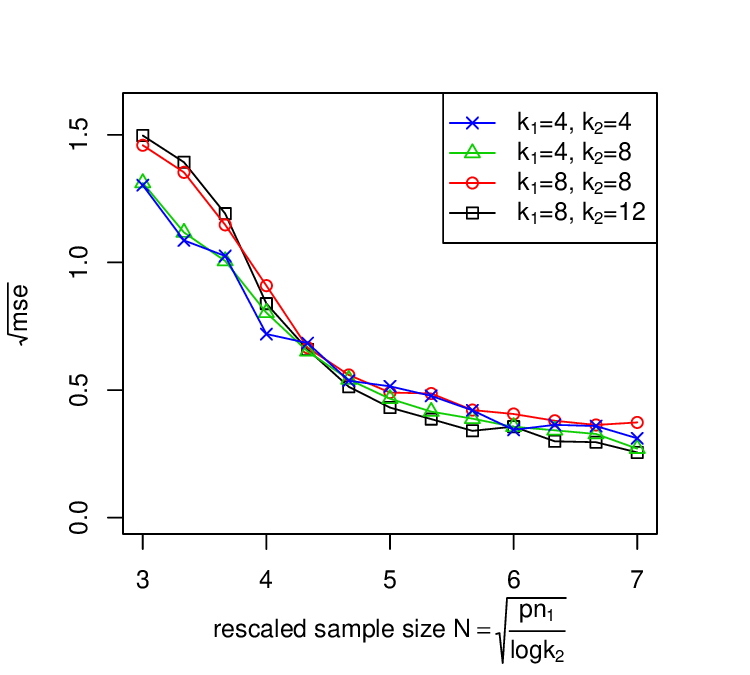}
  \caption{}
\end{subfigure}
\caption{\noindent Plots of  error $\frac{1}{\sqrt{n_1n_2}}\|\hat{\theta} - \theta\|$  when using our algorithm on biclustering data with gaussian noise. Each curve corresponds to a fixed $(k_1,k_2)$. (a) Plots of  error against $n_1$. (b) Plots of the same error against $\sqrt{pn_1/\log k_2}$. 
}
\label{fig:gaussian}
\end{figure}

\paragraph{Sparse Bernoulli case.} We also study recovery of sparse SBMs. We do the same simulation as the Bernoulli case except that we choose $Q=0.02 \mathbf{1}_k \mathbf{1}_k^T + 0.05 t \mathbf{I}_k$ for $t=1,2,3,4$. The results are shown in Figure \ref{fig:simulation3}.
\begin{figure}[h]
\centering
\begin{subfigure}{.5\textwidth}
  \centering
  \includegraphics[width=\linewidth]{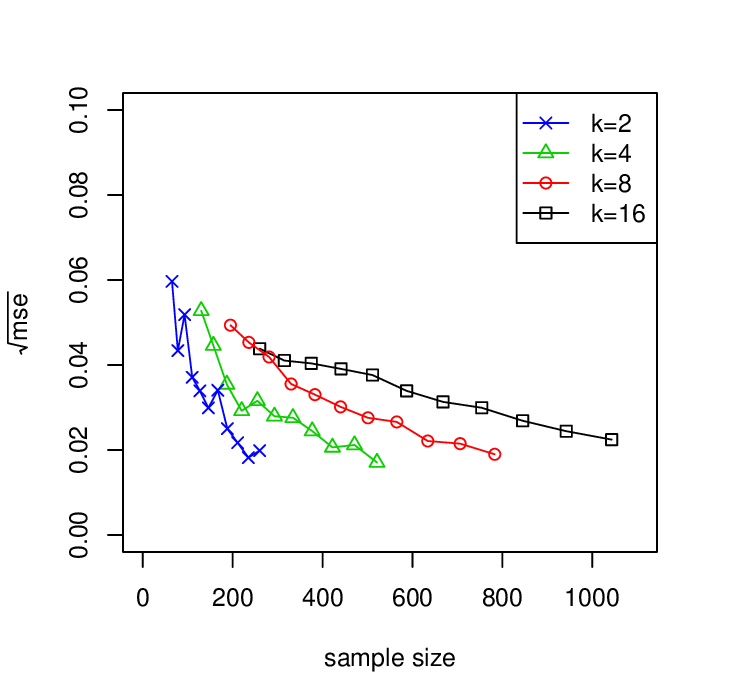}
  \caption{}
  \label{fig:sub1}
\end{subfigure}%
\begin{subfigure}{.5\textwidth}
  \centering
  \includegraphics[width=\linewidth]{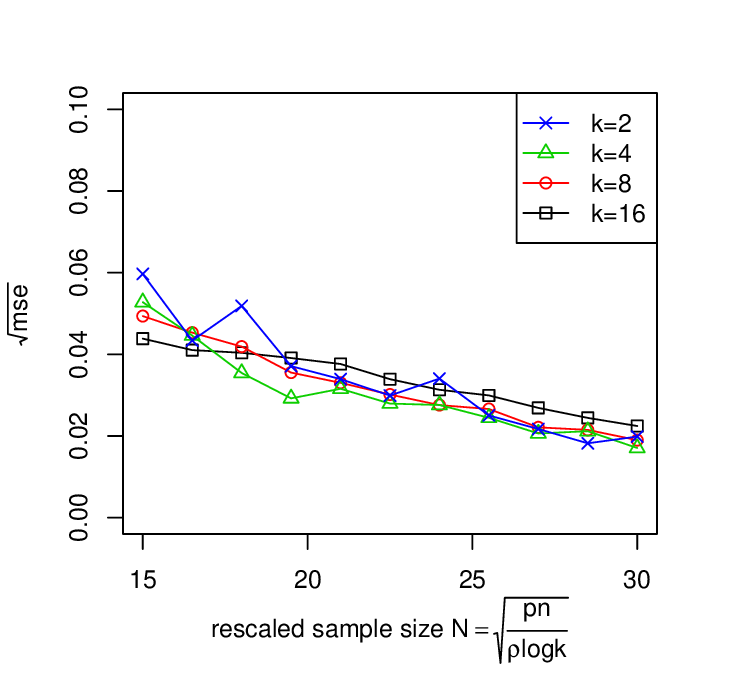}
  \caption{}
  \label{fig:sub2}
\end{subfigure}
\caption{\noindent Plots of error $\frac{1}{n} \|\hat{\theta} - \theta\|$  when using our algorithm on sparse SBM. Each curve corresponds to a fixed $k$. (a) Plots of error against the raw sample size $n$. (b) Plots of the same error against rescaled sample size $\sqrt{pn/\rho \log k}$. 
}
\label{fig:simulation3}
\end{figure}

\paragraph{Adaptation to unknown parameters.}
We use the 2-fold cross validation procedure proposed in Section \ref{sec:adapt} to adaptively choose the unknown number of clusters $k$ and the sparsity level $\rho$. We use the setting of sparse SBM with the number of block $k \in \{4,6\}$ and $Q=0.05 \mathbf{1}_k \mathbf{1}_k^T + 0.1 t \mathbf{I}_k$ for $t=1,2,3,4$. When running our algorithms, we search over all the $(k,\rho)$ pair for $k \in \{2,3,\cdots,8\}$ and $\rho \in \{0.2, 0.3, 0.4, 0.5\}$. In Table \ref{tab:adpk=4}, we report the errors for different configurations of $Q$. The first row is the error obtained by our adaptive procedure and the second row is the error using the true $k$ and $\rho$. Consistent with our Theorem \ref{thm:adaptive}, the error from the adaptive procedure is almost the same as the oracle error. 

\begin{table}[h!]
\begin{center}
\begin{tabular}{ c|cccccc } 
 \hline
rescaled sample size  & 6  & 12 & 18  & 24 \\
 \hline
(k=4) adaptive $\sqrt{mse}$ & 0.084 & 0.066 & 0.058 & 0.058 \\ 
	oracle $\sqrt{mse}$ & 0.085 & 0.069 & 0.060 & 0.053 \\ 
 \hline
 (k=6) adaptive $\sqrt{mse}$ & 0.074 & 0.061 & 0.051 & 0.050 \\ 
	oracle $\sqrt{mse}$ & 0.078 & 0.067 & 0.056 & 0.048 \\ 
 \hline
\end{tabular}
\end{center}
\caption{Errors of the adaptive procedure versus the oracle.}
\label{tab:adpk=4}
\end{table}

\paragraph{The effect of constraints.} 
The optimization (\ref{eq:LS}) and Algorithm \ref{alg:bl} involves the constraint $Q\in[l,u]^{k_1\times k_2}$. It is curious whether this constraint really helps reduce the error or merely an artifact of the proof.
We investigate the effect of this constraint on simulated data by comparing Algorithm \ref{alg:bl} with its variation without the constraint for both Gaussian case and sparse Bernoulli case. 
Panel (a) of Figure \ref{fig:constrain} shows the plots of sparse SBM with 8 communities. Panel (b) is the plots of Gaussian case with $(k_1,k_2)=(4,8)$. For both panels, when the rescaled sample size is small, the effect of constraint is significant, while as the rescaled sample size increases, the performance of two estimators become similar. 

\begin{figure}[h]
\centering
\begin{subfigure}{.5\textwidth}
  \centering
  \includegraphics[width=\linewidth]{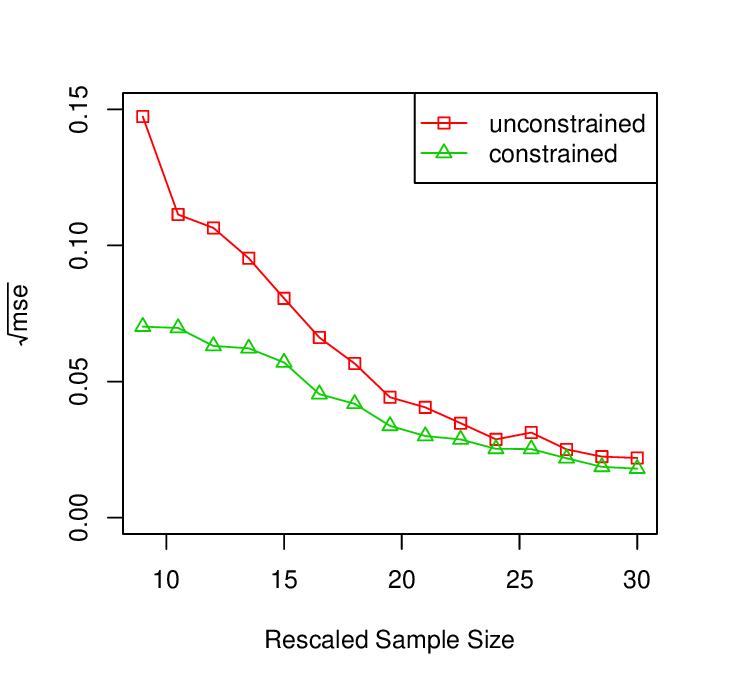}
  \caption{Sparse SBM with $k=8$}
  \label{fig:sub1}
\end{subfigure}%
\begin{subfigure}{.5\textwidth}
  \centering
  \includegraphics[width=\linewidth]{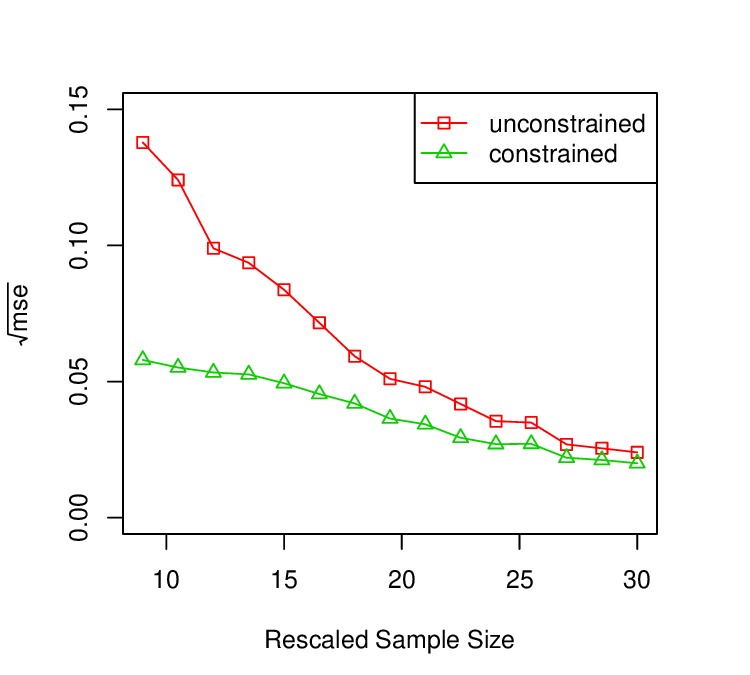}
  \caption{Gaussian biclustering with $(k_1, k_2)=(4,8)$}
  \label{fig:sub2}
\end{subfigure}
\caption{\noindent Constrained versus unconstrained least-squares}
\label{fig:constrain}
\end{figure}

\section{Discussion}\label{sec:discussion}

This paper studies the optimal rates of recovering a matrix with biclustering structure. While the recent progresses in high-dimensional estimation mainly focus on sparse and low rank structures, the study of biclustering structure does not gain much attention. This paper fills in the gap. 
In what follows, we discuss some key points of the paper and some possible future directions of research.

\textit{Difference from low-rankness.} 
A biclustering structure is implicitly low-rank. 
Therefore, we show that by exploring the stronger biclustering assumption, one can achieve better rates of convergence in estimation and completion. 
The minimax rates derived in this paper precisely characterize how much one can gain by taking advantage of this structure. 

\textit{Relation to other structures.} A natural question to investigate is whether there is similarity between the biclustering structure and the well-studied sparsity structure. The paper \cite{gao2015} gives a general theory of structured estimation in linear models that puts both sparse and biclustering structures in a unified theoretical framework. According to this general theory, the $k_1k_2$ part in the minimax rate is the complexity of parameter estimation and the $n_1\log k_1+n_2\log k_2$ part is the complexity of structure estimation.

\textit{Open problems.} The optimization problem (\ref{eq:LS}) is not convex, thus causing difficulty in devising a provably optimal polynomial-time algorithm. 
An open question is whether there is a convex relaxation of (\ref{eq:LS}) that can be solved efficiently without losing much statistical accuracy. 
Another interesting problem for future research is whether the objective function in (\ref{eq:LS}) can be extended beyond the least squares framework. 

\section{Proofs}\label{sec:pf}

\subsection{Proof of Theorem \ref{thm:main}}

Below, we focus on the proof for the asymmetric parameter space $\Theta_{k_1k_2}(M)$. The result for the symmetric parameter space $\Theta_{k}^s(M)$ can be obtained by letting $k_1=k_2$ and by taking care of the diagonal entries.
Since $\hat{\theta}\in \Theta_{k_1k_2}(M)$, there exists $\hat{z}_1\in[k_1]^{n_1}$, $\hat{z}_2\in[k_2]^{n_2}$ and $\hat{Q}\in[-M,M]^{k_1\times k_2}$ such that $\hat{\theta}_{ij}=\hat{Q}_{\hat{z}_1(i)\hat{z}_2(j)}$. For this $(\hat{z}_1,\hat{z}_2)$, we define a matrix $\tilde{\theta}$ by
$$\tilde{\theta}_{ij}= \frac{1}{|\hat{z}_1^{-1}(a)||\hat{z}_2^{-1}(b)|}\sum_{(i,j)\in\hat{z}_1^{-1}(a)\times \hat{z}_2^{-1}(b)}\theta_{ij},$$
for any $(i,j)\in\hat{z}_1^{-1}(a)\times \hat{z}_2^{-1}(b)$ and any $(a,b)\in[k_1]\times [k_2]$. To facilitate the proof, we need to following three lemmas, whose proofs are given in the supplementary material.

\begin{lemma} \label{lem:average}
For any constant $C'>0$, there exists a constant $C_1>0$ only depending on $C'$, such that
$$\norm{\hat{\theta}-\tilde{\theta}}^2 \leq  C_1 \frac{M^2\vee\sigma^2}{p} (k_1 k_2+n_1 \log k_1 + n_2 \log k_2),$$
with probability at least $1-\exp(-C'(k_1k_2+n_1 \log k_1 + n_2 \log k_2))$.
\end{lemma}

\begin{lemma} \label{lem:partition1}
For any constant $C'>0$, there exists a constant $C_2>0$ only depending on $C'$, such that the inequality $\norm{\tilde{\theta}-\theta}^2 \ge C_2 (M^2 \vee \sigma^2) (\dimnone \log \dimkone + \dimntwo \log \dimktwo)/p$ implies
$$\left|\iprod{\frac{\tilde{\theta}-\theta}{\norm{\tilde{\theta}-\theta}}}{Y-\theta}\right| \leq \sqrt{C_2\frac{M^2\vee\sigma^2}{p}(k_1k_2+\dimnone \log \dimkone + \dimntwo \log \dimktwo)},$$
with probability at least $1-\exp(-C'(k_1k_2+\dimnone \log \dimkone + \dimntwo \log \dimktwo))$.
\end{lemma}

\begin{lemma} \label{lem:partition2}
For any constant $C'>0$, there exists a constant $C_3>0$ only depending on $C'$, such that
$$\left|\iprod{\hat{\theta}-\tilde{\theta}}{Y-\theta}\right| \leq C_3 \frac{M^2\vee\sigma^2}{p} (\dimkone\dimktwo+\dimnone \log \dimkone + \dimntwo \log \dimktwo),$$
with probability at least $1-\exp(-C'(k_1k_2+\dimnone \log \dimkone + \dimntwo \log \dimktwo))$.
\end{lemma}

\begin{proof}[Proof of Theorem \ref{thm:main}]
Applying union bound, the results of Lemma \ref{lem:average}-\ref{lem:partition2} hold with probability at least $1-3\exp\left(-C'(k_1k_2+n_1\log k_1+n_2\log k_2)\right)$.
We consider the following two cases.
\paragraph{Case 1:}
$\norm{\tilde{\theta}-\theta}^2 \le C_2 (M^2 \vee \sigma^2) (k_1k_2+\dimnone \log \dimkone + \dimntwo \log \dimktwo)/p$. \\
Then we have
$$\norm{\hat{\theta}-\theta}^2\leq 2\norm{\hat{\theta}-\tilde{\theta}}^2+2\norm{\tilde{\theta}-\theta}^2\leq 2(C_1+C_2) \frac{M^2 \vee \sigma^2}{p} (k_1k_2+\dimnone \log \dimkone + \dimntwo \log \dimktwo)$$
by Lemma \ref{lem:average}.
\paragraph{Case 2:}
$\norm{\tilde{\theta}-\theta}^2 > C_2 (M^2 \vee \sigma^2) (k_1k_2+\dimnone \log \dimkone + \dimntwo \log \dimktwo)/p$. \\
By the definition of the estimator, we have $\norm{\hat{\theta}-Y}^2\leq\norm{\theta-Y}^2$. After rearrangement, we have
\begin{eqnarray*}
\nonumber \norm{\hat{\theta}-\theta}^2 &\leq& 2\iprod{\hat{\theta}-\theta}{Y-\theta} \\
\nonumber &=& 2 \iprod{\hat{\theta}-\tilde{\theta}}{Y-\theta}
+ 2 \iprod{\tilde{\theta}-\theta}{Y-\theta} \\
\nonumber &\leq& 2 \iprod{\hat{\theta}-\tilde{\theta}}{Y-\theta} + 2 \norm{\tilde{\theta}-\theta} \left|\iprod{\frac{\tilde{\theta}-\theta}{\norm{\tilde{\theta}-\theta}}}{Y-\theta}\right| \\
\nonumber &\le& 2 \iprod{\hat{\theta}-\tilde{\theta}}{Y-\theta} + 2 (\norm{\tilde{\theta}-\hat{\theta}} + \norm{\hat{\theta} - \theta}) \left|\iprod{\frac{\tilde{\theta}-\theta}{\norm{\tilde{\theta}-\theta}}}{Y-\theta}\right|  \\
&\leq& 2(C_2+C_3+\sqrt{C_1C_2})\frac{M^2 \vee \sigma^2}{p} (k_1k_2+\dimnone \log \dimkone + \dimntwo \log \dimktwo)+\frac{1}{2}\norm{\hat{\theta}-\theta}^2,
\end{eqnarray*}
which leads to the bound
$$\norm{\hat{\theta}-\theta}^2\leq 4(C_2+C_3+\sqrt{C_1C_2})\frac{M^2 \vee \sigma^2}{p} (k_1k_2+\dimnone \log \dimkone + \dimntwo \log \dimktwo).$$
Combining the two cases, we have
$$\norm{\hat{\theta}-\theta}^2\leq C\frac{M^2 \vee \sigma^2}{p} (k_1k_2+\dimnone \log \dimkone + \dimntwo \log \dimktwo),$$
with probability at least $1-3\exp\left(-C'(k_1k_2+n_1\log k_1+n_2\log k_2)\right)$ for $C=4(C_2+C_3+\sqrt{C_1C_2})\vee 2(C_1+C_2)$.
\end{proof}

\subsection{Proof of Theorem \ref{thm:adaptive}}

We first present a lemma for the tail behavior of sum of independent products of sub-Gaussian and Bernoulli random variables. Its proof is given in the supplementary material.
\begin{lemma}\label{lm:subexp}
Let $\{X_i\}$ be independent sub-Gaussian random variables with mean $\theta_i \in [-M, M]$ and $\E e^{\lambda (X_i - \theta_i)} \le e^{\lambda^2 \sigma^2/2}$. Let $\{E_i\}$ be independent Bernoulli random variables with mean $p$. Assume $\{X_i\}$ and $\{E_i\}$ are all independent. Then for $|\lambda| \le p/(M \vee \sigma)$ and $Y_i=X_iE_i/p$, we have
$$ \E e^{\lambda (Y_i-\theta_i)} \le  e^{2(M^2+2\sigma^2)\lambda^2/p}.$$
Moreover, for $\sum_{i=1}^{n} c_i^2=1$,
\begin{equation}\label{eq:bernstein1}
\Prob \left\{ \left|\sum_{i=1}^{n} c_{i} (Y_i - \theta_i)\right| \ge t \right\} \le 2 \exp \left \{ - \min \left(\frac{p t^2}{8(M^2+2\sigma^2)}, \frac{pt}{2(M \vee \sigma) \norm{c}_{\infty}} \right) \right \}
\end{equation}
for any $t>0$.
\end{lemma}

\begin{proof}[Proof of Theorem \ref{thm:adaptive}]
Consider the mean matrix $\theta$ that belongs to the space $\Theta_{k_1k_2}(M)$.
By the definition of $(\hat{k}_1,\hat{k}_2,\hat{M})$, we have $\norm{\hat{\theta}_{\hat{k}_1\hat{k}_2\hat{M}}^{\Delta}-Y^{\Delta^c}}^2_{\Delta^c}\leq \norm{\hat{\theta}_{k_1k_2m}^{\Delta}-Y^{\Delta^c}}^2_{\Delta^c}$, where $k_1$ and $k_2$ are true numbers of row and column clusters and $m$ is chosen to be the smallest element in $\mathcal{M}$ that is no smaller than $M$.  After rearrangement, we have
\begin{eqnarray*}
&& \norm{\hat{\theta}_{\hat{k}_1\hat{k}_2\hat{M}}^{\Delta}-\theta}_{\Delta^c}^2 \\
&\leq& \norm{\hat{\theta}_{k_1k_2m}^{\Delta}-\theta}_{\Delta^c}^2 + 2\norm{\hat{\theta}_{\hat{k}_1\hat{k}_2\hat{M}}^{\Delta}-\hat{\theta}_{k_1k_2m}^{\Delta}}_{\Delta^c}\iprod{\frac{\hat{\theta}_{\hat{k}_1\hat{k}_2\hat{M}}^{\Delta}-\hat{\theta}_{k_1k_2m}^{\Delta}}{\norm{\hat{\theta}_{\hat{k}_1\hat{k}_2\hat{M}}^{\Delta}-\hat{\theta}_{k_1k_2m}^{\Delta}}_{\Delta^c}}}{Y^{\Delta^c}-\theta}_{\Delta^c} \\
&\leq& \norm{\hat{\theta}_{k_1k_2m}^{\Delta}-\theta}_{\Delta^c}^2 + 2\norm{\hat{\theta}_{\hat{k}_1\hat{k}_2\hat{M}}^{\Delta}-\hat{\theta}_{k_1k_2m}^{\Delta}}_{\Delta^c}\max_{(l_1,l_2)\in[n_1]\times [n_2]}\left|\iprod{\frac{\hat{\theta}_{l_1l_2h}^{\Delta}-\hat{\theta}_{k_1k_2m}^{\Delta}}{\norm{\hat{\theta}_{l_1l_2h}^{\Delta}-\hat{\theta}_{k_1k_2m}^{\Delta}}_{\Delta^c}}}{Y^{\Delta^c}-\theta}_{\Delta^c}\right|.
\end{eqnarray*}
By Lemma \ref{lm:subexp} and the independence structure, we have
$$\max_{(l_1,l_2,h)\in[n_1]\times [n_2]\times \mathcal{M}}\left|\iprod{\frac{\hat{\theta}_{l_1l_2h}^{\Delta}-\hat{\theta}_{k_1k_2m}^{\Delta}}{\norm{\hat{\theta}_{l_1l_2h}^{\Delta}-\hat{\theta}_{k_1k_2m}^{\Delta}}_{\Delta^c}}}{Y^{\Delta^c}-\theta}_{\Delta^c}\right|\leq C(M\vee\sigma)\frac{\log(n_1+n_2)}{p},$$
with probability at least $1-(n_1n_2)^{-C'}$. Using triangle inequality and Cauchy-Schwarz inequality, we have
$$\norm{\hat{\theta}_{\hat{k}_1\hat{k}_2\hat{M}}^{\Delta}-\theta}_{\Delta^c}^2\leq \frac{3}{2}\norm{\hat{\theta}_{k_1k_2m}^{\Delta}-\theta}_{\Delta^c}^2+\frac{1}{2}\norm{\hat{\theta}_{\hat{k}_1\hat{k}_2\hat{M}}^{\Delta}-\theta}_{\Delta^c}^2+4C^2(M^2\vee\sigma^2)\left(\frac{\log(n_1+n_2)}{p}\right)^2.$$
By rearranging the above inequality, we have
$$\norm{\hat{\theta}_{\hat{k}_1\hat{k}_2\hat{M}}^{\Delta}-\theta}_{\Delta^c}^2\leq 3\norm{\hat{\theta}_{k_1k_2m}^{\Delta}-\theta}_{\Delta^c}^2+8C^2(M^2\vee\sigma^2)\left(\frac{\log(n_1+n_2)}{p}\right)^2.$$
A symmetric argument leads to
$$\norm{\hat{\theta}_{\hat{k}_1\hat{k}_2\hat{M}}^{\Delta^c}-\theta}_{\Delta}^2\leq 3\norm{\hat{\theta}_{k_1k_2m}^{\Delta^c}-\theta}_{\Delta}^2+8C^2(M^2\vee\sigma^2)\left(\frac{\log(n_1+n_2)}{p}\right)^2.$$
Summing up the above two inequalities, we have
\begin{equation}
\norm{\hat{\theta}-\theta}^2\leq 3\norm{\hat{\theta}_{k_1k_2m}^{\Delta}-\theta}^2+3\norm{\hat{\theta}_{k_1k_2m}^{\Delta^c}-\theta}^2+16C^2(M^2\vee\sigma^2)\left(\frac{\log(n_1+n_2)}{p}\right)^2.\label{eq:cena}
\end{equation}
Using Theorem \ref{thm:main} to bound $\norm{\hat{\theta}_{k_1k_2m}^{\Delta}-\theta}^2$ and $\norm{\hat{\theta}_{k_1k_2m}^{\Delta^c}-\theta}^2$ can be bounded by $C\frac{m^2\vee\sigma^2}{p}(k_1k_2+n_1\log  k_1+n_2\log k_2)$. Given that $m=M\left(1+\frac{m-M}{M}\right)\leq 2M$ by the choice of $m$, the proof is complete.
\end{proof}

\subsection{Proof of Theorem \ref{thm:p}}

Recall the augmented data $Y_{ij}=X_{ij}E_{ij}/p$. Define $\mathcal{Y}_{ij}=X_{ij}E_{ij}/\hat{p}$.
Let us give two lemmas to facilitate the proof.
\begin{lemma}\label{lem:aaa}
Assume $p\gtrsim \frac{\log(n_1+n_2)}{n_1n_2}$. For any $C'>0$, there is some constant $C>0$
 such that
$$\norm{Y-\mathcal{Y}}^2\leq C\left[M^2+\sigma^2\log(n_1+n_2)\right]\frac{\log(n_1+n_2)}{p^2},$$
with probability at least $1-(n_1n_2)^{-C'}$.
\end{lemma}
\begin{lemma}\label{lem:aaaa}
The inequalities in Lemma \ref{lem:average}-\ref{lem:partition2} continue to hold with bounds
$$C_1 \frac{M^2\vee\sigma^2}{p} (k_1 k_2+n_1 \log k_1 + n_2 \log k_2)+2\norm{Y-\mathcal{Y}}^2,$$
$$\sqrt{C_2\frac{M^2\vee\sigma^2}{p}(k_1k_2+\dimnone \log \dimkone + \dimntwo \log \dimktwo)}+\norm{Y-\mathcal{Y}},$$
and
$$C_3 \frac{M^2\vee\sigma^2}{p} (\dimkone\dimktwo+\dimnone \log \dimkone + \dimntwo \log \dimktwo)+\norm{\hat{\theta}-\tilde{\theta}}\norm{Y-\mathcal{Y}},$$
respectively.
\end{lemma}
\begin{proof}[Proof of Theorem \ref{thm:p}]
The proof is similar to that of Theorem \ref{thm:main}. We only need to replace Lemma \ref{lem:average}-\ref{lem:partition2} by Lemma \ref{lem:aaa} and Lemma \ref{lem:aaaa} to get the desired result.
\end{proof}

\subsection{Proofs of Theorem \ref{thm:lower-Gaussian} and Theorem \ref{thm:lower-SBM}}

This section gives proofs of the minimax lower bounds. We first introduce some notation. For any probability measures $\mathbb{P},\mathbb{Q}$, define the Kullback--Leibler divergence by $D(\mathbb{P}||\mathbb{Q})=\int\left(\log\frac{d\mathbb{P}}{d\mathbb{Q}}\right)d\mathbb{P}$. The chi-squared divergence is defined by $\chi^2(\mathbb{P}||\mathbb{Q})=\int\left(\frac{d\mathbb{P}}{d\mathbb{Q}}\right)d\mathbb{P}-1$. The main tool we will use is the following proposition.
\begin{proposition} \label{prop:fano}
Let $(\Xi,\ell)$ be a metric space and $\{\mathbb{P}_{\xi}:\xi\in\Xi\}$ be a collection of probability measures. For any totally bounded $T\subset\Xi$, define the Kullback-Leibler diameter and the chi-squared diameter of $T$ by
$$d_{\text{KL}}(T)=\sup_{\xi,\xi'\in T}D(\mathbb{P}_{\xi}||\mathbb{P}_{\xi'}),\quad d_{\chi^2}(T)=\sup_{\xi,\xi'\in T}\chi^2(\mathbb{P}_{\xi}||\mathbb{P}_{\xi'}).$$
Then
\begin{eqnarray}
\label{eq:fanoKL}&&\inf_{\hat{\xi}}\sup_{\xi\in\Xi}\mathbb{P}_{\xi}\left\{\ell^2\Big(\hat{\xi}(X),\xi\Big)\geq\frac{\epsilon^2}{4}\right\} \geq 1-\frac{d_{\text{KL}}(T)+\log 2}{\log\mathcal{M}(\epsilon,T,\ell)}, \\
\label{eq:fanochi2}&&\inf_{\hat{\xi}}\sup_{\xi\in\Xi}\mathbb{P}_{\xi}\left\{\ell^2\Big(\hat{\xi}(X),\xi\Big)\geq\frac{\epsilon^2}{4}\right\} \geq 1-\frac{1}{\mathcal{M}(\epsilon,T,\ell)}-\sqrt{\frac{d_{\chi^2}(T)}{\mathcal{M}(\epsilon,T,\ell)}},
\end{eqnarray}
for any $\epsilon>0$, where the packing number $\mathcal{M}(\epsilon,T,\ell)$ is the largest number of points in $T$ that are at least $\epsilon$ away from each other.
\end{proposition}
The inequality (\ref{eq:fanoKL}) is the classical Fano's inequality. The version we present here is by \cite{yu97}. The inequality (\ref{eq:fanochi2}) is a generalization of the classical Fano's inequality by using chi-squared divergence instead of KL divergence. It is due to \cite{guntuboyina11}. 

The following proposition bounds the KL divergence and the chi-squared divergence for both Gaussian and Bernoulli models.
\begin{proposition}\label{prop:KLchi}
For the Gaussian model, we have
$$D\left(\mathbb{P}_{(\theta,\sigma^2,p)}||\mathbb{P}_{(\theta',\sigma^2,p)}\right)\leq\frac{p}{2\sigma^2}\norm{\theta-\theta'}^2,\quad \chi^2\left(\mathbb{P}_{(\theta,\sigma^2,p)}||\mathbb{P}_{(\theta',\sigma^2,p)}\right)\leq \exp\left(\frac{p}{\sigma^2}\norm{\theta-\theta'}^2\right)-1.$$
For the Bernoulli model with any $\theta,\theta'\in[\rho/2,3\rho/4]^{n_1\times n_2}$, we have
$$D\left(\mathbb{P}_{(\theta,p)}||\mathbb{P}_{(\theta',p)}\right)\leq \frac{8p}{\rho}\norm{\theta-\theta'}^2,\quad \chi^2\left(\mathbb{P}_{(\theta,p)}||\mathbb{P}_{(\theta',p)}\right)\leq\exp\left(\frac{8p}{\rho}\norm{\theta-\theta'}^2\right)-1.$$ 
\end{proposition}

Finally, we need the following Varshamov--Gilbert bound. The version we present here is due to \cite[Lemma 4.7]{massart2007}.
\begin{lemma}\label{lem:VG}
There exists a subset $\{\omega_1,...,\omega_N\}\subset \{0,1\}^d$ such that
\begin{equation}
H(\omega_i, \omega_j) \triangleq \norm{\omega_i-\omega_j}^2\geq\frac{d}{4},\quad\text{for any }i\neq j\in[N],\label{eq:defH}
\end{equation}
for some $N\geq \exp {(d/8)}$. 
\end{lemma}

\begin{proof}[Proof of Theorem \ref{thm:lower-Gaussian}]
We focus on the proof for the asymmetric parameter space $\Theta_{k_1k_2}(M)$. The result for the symmetric parameter space $\Theta_{k}^s(M)$ can be obtained by letting $k_1=k_2$ and by taking care of the diagonal entries.
Let us assume $n_1/k_1$ and $n_2/k_2$ are integers without loss of generality. We first derive the lower bound for the nonparametric rate $\sigma^2 k_1k_2/p$. Let us fix the labels by $z_1(i)=\ceil{ik_1/n_1}$ and $z_2(j)=\ceil{jk_2/n_2}$. For any $\omega\in\{0,1\}^{k_1\times k_2}$, define
\begin{equation}
Q_{ab}^{\omega}=c\sqrt{\frac{\sigma^2 k_1k_2}{pn_1n_2}}\omega_{ab}.\label{eq:lf-Q}
\end{equation}
By Lemma \ref{lem:VG}, there exists some $T\subset\{0,1\}^{k_1k_2}$ such that $|T|\geq \exp(k_1k_2/8)$ and $H(\omega,\omega')\geq k_1k_2/4$ for any $\omega,\omega\in T$ and $\omega\neq\omega'$. We construct the subspace
$$\Theta(z_1,z_2,T)=\left\{\theta\in\mathbb{R}^{n_1\times n_2}: \theta_{ij}=Q^{\omega}_{z_1(i)z_2(j)}, \omega\in T\right\}.$$
By Proposition \ref{prop:KLchi}, we have
$$\sup_{\theta,\theta'\in\Theta(z_1,z_2,T)}\chi^2\left(\mathbb{P}_{(\theta,\sigma^2,p)}||\mathbb{P}_{(\theta',\sigma^2,p)}\right)\leq \exp\left(c^2k_1k_2\right).$$
For any two different  $\theta$ and $\theta'$ in $\Theta(z_1,z_2,T)$ associated with $\omega,\omega'\in T$, we have
$$\norm{\theta-\theta'}^2\geq \frac{c^2\sigma^2}{p}H(\omega,\omega')\geq \frac{c^2\sigma^2}{4p}k_1k_2.$$
Therefore, $\mathcal{M}\left(\sqrt{\frac{c^2\sigma^2}{4p}k_1k_2},\Theta(z_1,z_2,T),\norm{\cdot}\right)\geq \exp(k_1k_2/8)$. Using (\ref{eq:fanochi2}) with an appropriate $c$, we have obtained the rate $\frac{\sigma^2}{p}k_1k_2$ in the lower bound.

Now let us derive the clustering rate $\sigma^2n_2\log k_2/p$. Let us pick $\omega_1,...,\omega_{k_2}\in\{0,1\}^{k_1}$ such that $H(\omega_a,\omega_b)\geq\frac{k_1}{4}$ for all $a\neq b$. By Lemma \ref{lem:VG}, this is possible when $\exp(k_1/8)\geq k_2$. Then, define
\begin{equation}
Q_{*a}=c\sqrt{\frac{\sigma^2 n_2\log k_2}{pn_1n_2}}\omega_a.\label{eq:lf-Q2}
\end{equation}
Define $z_1$ by $z_1(i)=\ceil{ik_1/n_1}$. Fix $Q$ and $z_1$ and we are gong to let $z_2$ vary. Select a set $\mathcal{Z}_2\subset [k_2]^{n_2}$ such that $|\mathcal{Z}_2|\geq \exp(Cn_2\log k_2)$ and $H(z_2,z_2')\geq \frac{n_2}{6}$ for any $z_2,z_2'\in\mathcal{Z}_k$ and $z_2\neq z_2'$. The existence of such $\mathcal{Z}_2$ is proved by \cite{gao2014rate}. Then, the subspace we consider is
$$\Theta(z_1,\mathcal{Z}_2,Q)=\left\{\theta\in\mathbb{R}^{n_1\times n_2}:\theta_{ij}=Q_{z_1(i)z_2(j)}, z_2\in\mathcal{Z}_2\right\}.$$
By Proposition \ref{prop:KLchi}, we have
$$\sup_{\theta,\theta'\in\Theta(z_1,\mathcal{Z}_2,Q)}D\left(\mathbb{P}_{(\theta,\sigma^2,p)}||\mathbb{P}_{(\theta',\sigma^2,p)}\right)\leq c^2n_2\log k_2.$$
For any two different $\theta$ and $\theta'$ in $\Theta(z_1,\mathcal{Z}_2,Q)$ associated with $z_2,z_2'\in\mathcal{Z}_2$, we have
$$\norm{\theta-\theta'}^2=\sum_{j=1}^{n_2}\norm{\theta_{*z_2(j)}-\theta'_{*z_2(j)}}^2\geq H(z_2,z_2')\frac{c^2\sigma^2 n_2\log k_2}{pn_1n_2}\frac{n_1}{4}\geq \frac{c^2\sigma^2 n_2\log k_2}{24p}.$$
Therefore, $\mathcal{M}\left(\sqrt{\frac{c^2\sigma^2 n_2\log k_2}{24p}}, \Theta(z_1,\mathcal{Z}_2,Q),\norm{\cdot}\right)\geq\exp(Cn_2\log k_2)$. 
Using (\ref{eq:fanoKL}) with some appropriate $c$, we obtain the lower bound $\frac{\sigma^2 n_2\log k_2}{p}$. 

A symmetric argument gives the rate $\frac{\sigma^2 n_1\log k_1}{p}$. Combining the three parts using the same argument in \cite{gao2014rate}, the proof is complete.
\end{proof}

\begin{proof}[Proof of Theorem \ref{thm:lower-SBM}]
The proof is similar to that of Theorem \ref{thm:lower-Gaussian}. The only differences are (\ref{eq:lf-Q}) replaced by
$$Q_{ab}^{\omega}=\frac{1}{2}\rho+\left(c\sqrt{\frac{\rho k^2}{pn^2}}\wedge \frac{1}{2}\rho\right)\omega_{ab}$$
and (\ref{eq:lf-Q2}) replaced by
$$Q_{*a}=\frac{1}{2}\rho+\left(c\sqrt{\frac{\rho \log k}{pn}}\wedge\frac{1}{2}\rho\right)\omega_{a}.$$
It is easy to check that the constructed subspaces are subsets of $\Theta_k^+(\rho)$. Then, a symmetric modification of the proof of Theorem \ref{thm:lower-Gaussian} leads to the desired conclusion.
\end{proof}

\subsection{Proofs of Corollary \ref{cor:spg} and Corollary \ref{cor:spg-adaptive}}

The result of Corollary \ref{cor:spg} can be derived through a standard bias-variance trade-off argument by combining Corollary \ref{cor:SBM} and Lemma 2.1 in \cite{gao2014rate}. The result of Corollary \ref{cor:spg-adaptive} follows Theorem \ref{thm:adaptive}. By studying the proof of Theorem \ref{thm:adaptive}, (\ref{eq:cena}) holds for all $k$. Choosing the best $k$ to trade-off bias and variance gives the result of Corollary \ref{cor:spg-adaptive}. We omit the details here.

\appendix


\section{Proofs of auxiliary results} \label{sec:auxiliary}

In this section, we give proofs of Lemma \ref{lem:average}-\ref{lem:aaa}. We first introduce some notation. Define the set
$$ \mathcal{Z}_{k_1k_2}=\{ z=(z_1, z_2): z_1\in[k_1]^{n_1}, z_2\in[k_2]^{n_2}\}.$$
For a matrix $G \in \mathbb{R}^{\dimnone \times \dimntwo}$ and some $z=(z_1,z_2)\in\mathcal{Z}_{k_1k_2}$, define 
$$ \bar{G}_{ab}(z) = \frac{1}{|z_1^{-1}(a)| |z_2^{-1}(b)|} \sum_{(i,j) \in z^{-1}_1(a)\times z^{-1}_2(b)} G_{ij},$$
for all $a \in [k_1], b \in [k_2]$. To facilitate the proof, we need the following two results.

\begin{proposition}\label{prop:ls}
For the estimator $\hat{\theta}_{ij}=\hat{Q}_{\hat{z}_1(i)\hat{z}_2(j)}$, we have
$$\hat{Q}_{ab} = \text{sign}(\bar{Y}_{ab}(\hat{z}))\left(|\bar{Y}_{ab}(\hat{z})| \wedge M\right),$$
for all $a \in [k_1], b \in [k_2]$.
\end{proposition}

\begin{lemma}\label{lm:subexV}
Under the setting of Lemma \ref{lm:subexp}, define $S=\frac{1}{\sqrt{n}}\sum_{i=1}^{n} (Y_i-\theta_i)$ and $\tau=2(M^2 + 2 \sigma^2)/(M \vee \sigma)$. Then we have the following results:
\begin{enumerate}
\item[a.] Let $T=S \mathbf{1} \{|S| \le \tau \sqrt{n} \}$, then $\E e^{p T^2/(8(M^2+2\sigma^2))} \le 5$;
\item[b.] Let $R=\tau \sqrt{n} |S| \mathbf{1} \{|S| > \tau \sqrt{n}\}$, then $\E e^{pR/(8(M^2+2\sigma^2))} \le 9$.
\end{enumerate}
 \end{lemma}
\begin{proof}
By (\ref{eq:bernstein1}),
\begin{equation*}
\Prob \Big( |S| > t \Big)  \le  4\exp \left \{ - \min \left( \frac{p t^2}{4(M^2+2\sigma^2)}, \frac{\sqrt{n}pt}{2(M \vee \sigma)}\right) \right \}.
\end{equation*}
Then
\begin{eqnarray*}
\E e^{\lambda T^2} &=& \int_{0}^\infty \Prob \Big( e^{\lambda T^2} > u \Big) d u \leq 1 + \int_{1}^\infty \Prob \left( |T| > \sqrt{\frac{\log u}{\lambda}} \right) d u \\
&=& 1 + \int_{1}^{e^{\lambda \tau^2 n}} \Prob \left( |S| > \sqrt{\frac{\log u}{\lambda}} \right) d u =  1 + 4\int_{1}^{e^{\lambda \tau^2 n}}  u^{-p/(4\lambda(M^2+2\sigma^2))} d u.
\end{eqnarray*}
Choosing $\lambda=p/(8(M^2+2\sigma^2))$, we get $\E e^{p T^2/(8(M^2+2\sigma^2))} \le 5$. 
We proceed to prove the second claim.
\begin{eqnarray*}
\E e^{\lambda R} &=& \Prob (R=0) + \Prob (R>0) \E [e^{\lambda R}|R>0] \\
&=& \Prob (R=0) + \Prob (R>0) \int_{0}^\infty \Prob (e^{\lambda R} > u | R>0 ) d u\\
&=& \Prob (R=0) +  \int_{0}^\infty \Prob (e^{\lambda R} > u , R>0 ) d u\\
&\leq& \Prob (R=0) +  \Prob (R>0)e^{\lambda \tau^2 n} + \int_{e^{\lambda \tau^2 n}}^\infty \Prob (e^{\lambda R} > u) d u\\
&\le& 1 + 4e^{- p \tau^2 n/(4(M^2+2\sigma^2))+\lambda \tau^2 n} + \int_{e^{\lambda\tau^2n}}^\infty \Prob ( e^{\sqrt{n} \lambda \tau |S|} > u) d u \\
&=& 1 + 4e^{- p \tau^2 n/(4(M^2+3\sigma^2))+\lambda \tau^2 n} + 4\int_{e^{\lambda\tau^2 n}}^\infty u^{-p/(2\lambda\tau(M\vee \sigma))} du 
\end{eqnarray*}
Choosing $\lambda = p/(8(M^2+2\sigma^2))$, we get
$ \E e^{ pR/(8(M^2+2\sigma^2))} \le 9$.
\end{proof}

\begin{proof}[Proof of Lemma \ref{lem:average}]
By the definitions of $\hat{\theta}_{ij}$ and $\tilde{\theta}_{ij}$ and Proposition \ref{prop:ls}, we have
\begin{align}\nonumber
\hat{\theta}_{ij}-\tilde{\theta}_{ij}=&
  \begin{cases}
    M-\bar{\theta}_{ab}(\hat{z}),       & \text{if }  \bar{Y}_{ab}(\hat{z}) \geq M; \\
    \bar{Y}_{ab}(\hat{z})-\bar{\theta}_{ab}(\hat{z}),       & \text{if } -M \le \bar{Y}_{ab}(\hat{z}) < M;  \\
    -M-\bar{\theta}_{ab}(\hat{z}),       & \text{if }  \bar{Y}_{ab}(\hat{z}) < -M
  \end{cases}
\end{align}
for any $(i,j)\in \hat{z}_1^{-1}(a)\times \hat{z}_2^{-1}(b)$. Define $W=Y-\theta$,  and it is easy to check that 
$$
|\hat{\theta}_{ij}-\tilde{\theta}_{ij}| \le |\bar{W}_{ab}(\hat{z})| \wedge 2M \le |\bar{W}_{ab}(\hat{z})| \wedge \tau,
$$
where $\hat{z}=(\hat{z}_1,\hat{z}_2)$ and  $\tau$ is defined in Lemma \ref{lm:subexV}.
Then
\begin{eqnarray}
\nonumber \norm{\hat{\theta}-\tilde{\theta}}^2 &\le & \sum_{a\in [\dimkone], b \in [\dimktwo]}  \left|\hat{z}_1^{-1}(a)\right|\left|\hat{z}_2^{-1}(b)\right| \Big ( |\bar{W}_{ab}(\hat{z})| \wedge \tau \Big )^2 \\
\label{eq:upperB}  &\le & \max_{z\in\mathcal{Z}_{k_1k_2}} \sum_{a\in [\dimkone], b \in [\dimktwo]}  \left|z_1^{-1}(a)\right|\left|z_2^{-1}(b)\right| \Big (|\bar{W}_{ab}(z)| \wedge \tau \Big )^2.
\end{eqnarray}
For any $a\in [\dimkone], b \in [\dimktwo]$ and $z_1\in[k_1]^{n_1},z_2\in[k_2]^{n_2}$, define $n_1(a)=\left | z_1^{-1} (a) \right |$, $n_2(b)=\left | z_2^{-1} (b) \right |$ and $$V_{ab}(z) = \sqrt{n_1(a) n_2(b)} |\bar{W}_{ab}(z)| \mathbf{1}\{|\bar{W}_{ab}(z)|\le \tau\}, $$ $$R_{ab}(z) = n_1(a) n_2(b) \tau |\bar{W}_{ab}(z)| \mathbf{1}\{|\bar{W}_{ab}(z)|> \tau\}.$$ Then,  
\begin{equation} \label{eq:lemma1bound}
\norm{\hat{\theta}-\tilde{\theta}}^2 \le \max_{z\in\mathcal{Z}_{k_1k_2}} \sum_{a\in [\dimkone], b \in [\dimktwo]} \left( V_{ab}^2(z) + R_{ab}(z) \right).
\end{equation}
By Markov's inequality and Lemma \ref{lm:subexV}, we have
\begin{eqnarray*}
\nonumber \Prob \left( \sum_{a\in [\dimkone], b \in [\dimktwo]} V^2_{ab}(z) > t \right) &\le& e^{-p t/(8(M^2+2\sigma^2))} \prod _{a\in [\dimkone], b \in [\dimktwo]} e^{p {V}^2_{ab}(z)/(8(M^2+2\sigma^2))} \\  
&\le& \exp\left\{ - \frac{pt}{8(M^2+2\sigma^2)} + \dimkone \dimktwo \log 5 \right\},
\end{eqnarray*}
and
\begin{eqnarray*}
\nonumber \Prob \left( \sum_{a\in [\dimkone], b \in [\dimktwo]} R_{ab}(z) > t \right) &\le& e^{-p t/(8(M^2+2\sigma^2))} \prod _{a\in [\dimkone], b \in [\dimktwo]} e^{p R_{ab}(z)/(8(M^2+2\sigma^2))} \\  
&\le& \exp\left\{ - \frac{pt}{8(M^2+2\sigma^2)} + \dimkone \dimktwo \log 9 \right\},
\end{eqnarray*}
Applying union bound and using the fact that $\log|[k_1]^n|+\log|[k_2]^{n_2}|= \dimnone \log \dimkone + \dimntwo \log \dimktwo$, 
\begin{equation*}
\mathbb{P} \left( \max_{z\in\mathcal{Z}_{k_1k_2}}\sum_{a\in [\dimkone], b \in [\dimktwo]} V^2_{ab}(z) > t \right) \le \exp \left\{ - \frac{pt}{8(M^2+2\sigma^2)} + \dimkone \dimktwo \log 5 + \dimnone \log \dimkone + \dimntwo \log \dimktwo \right\}.
\end{equation*}
For any given constant $C'>0$, we choose $t  = C_1 \frac{M^2\vee\sigma^2}{p} (\dimkone \dimktwo + \dimnone \log \dimkone + \dimntwo \log \dimktwo)$ for some sufficiently large $C_1>0$ to obtain
\begin{equation} \label{eq:upperV} 
\max_{z\in\mathcal{Z}_{k_1k_2}}  \sum_{a\in [\dimkone], b \in [\dimktwo]} V^2_{ab}(z) \leq  C_1 \frac{M^2\vee\sigma^2}{p} (\dimkone \dimktwo + \dimnone \log \dimkone + \dimntwo \log \dimktwo)
\end{equation}
with probability at least $1-\exp\left(-C' (k_1k_2+\dimnone \log \dimkone + \dimntwo \log \dimktwo)\right)$. Similarly, for some sufficiently large $C_2>0$, we have
\begin{equation} \label{eq:upperR} 
\max_{z\in\mathcal{Z}_{k_1k_2}}  \sum_{a\in [\dimkone], b \in [\dimktwo]} R_{ab}(z) \leq  C_2 \frac{M^2\vee\sigma^2}{p} (\dimkone \dimktwo + \dimnone \log \dimkone + \dimntwo \log \dimktwo)
\end{equation}
with probability at least $1-\exp\left(-C' (k_1k_2+\dimnone \log \dimkone + \dimntwo \log \dimktwo)\right)$. Plugging (\ref{eq:upperV}) and (\ref{eq:upperR}) into (\ref{eq:lemma1bound}), we complete the proof.
\end{proof}

\begin{proof}[Proof of Lemma \ref{lem:partition1}]
Note that
$$\tilde{\theta}_{ij}-\theta_{ij}=\sum_{a\in [\dimkone], b \in [\dimktwo]}\bar{\theta}_{ab}(\hat{z})\mathbf{1}\{(i,j)\in {\hat{z}_1}^{-1}(a)\times {\hat{z}_2}^{-1}(b)\}-\theta_{ij}$$
is a function of $\hat{z}_1$ and $\hat{z}_2$. Then we have
$$\left|\sum_{ij}\frac{\tilde{\theta}_{ij}-\theta_{ij}}{\sqrt{\sum_{ij}(\tilde{\theta}_{ij}-\theta_{ij})^2}}(Y_{ij}-\theta_{ij})\right|\leq \max_{z\in\mathcal{Z}_{k_1k_2}}\left|\sum_{ij}\gamma_{ij}(z)(Y_{ij}-\theta_{ij})\right|,$$
where
$$\gamma_{ij}(z) \propto  \sum_{a\in [\dimkone], b \in [\dimktwo]}\bar{\theta}_{ab}(z)\mathbf{1}\{(i,j)\in {z_1}^{-1}(a)\times {z_2}^{-1}(b)\}-\theta_{ij} $$ 
satisfies $\sum_{ij}\gamma_{ij}(z)^2=1.$
Consider the event $\norm{\tilde{\theta}-\theta}^2 \ge C_2(M^2 \vee \sigma^2)(k_1k_2+\dimnone \log \dimkone + \dimntwo \log \dimktwo)/p$ for some $C_2$ to be specified later, we have 
\begin{equation*}
|\gamma_{ij}(z)| \le \frac{2M}{\norm{\tilde{\theta}-\theta}} \le \sqrt{\frac{4M^2 p}{C_2(M^2 \vee \sigma^2) (k_1k_2+\dimnone \log \dimkone + \dimntwo \log \dimktwo)}}.
\end{equation*}
By Lemma \ref{lm:subexp} and union bound, we have
\begin{eqnarray*}
&& \mathbb{P}\left(\max_{z\in\mathcal{Z}_{k_1k_2}}\left|\sum_{ij}\gamma_{ij}(z)(Y_{ij}-\theta_{ij})\right|>t\right) \\
&\leq& \sum_{z_1\in[k_1]^{n_1},z_2\in[k_2]^{n_2}} \mathbb{P}\left(\left|\sum_{ij}\gamma_{ij}(z)(Y_{ij}-\theta_{ij})\right|>t\right) \\
&\leq&\exp\left(-C'\left(k_1k_2+n_1\log k_1+n_2\log k_2\right)\right),
\end{eqnarray*}
by setting $t = \sqrt{C_2(M^2\vee\sigma^2)(k_1k_2+\dimnone \log \dimkone + \dimntwo \log \dimktwo)/p}$ for some sufficiently large $C_2$ depending on $C'$. Thus, the lemma is proved.
\end{proof}

\begin{proof}[Proof of Lemma \ref{lem:partition2}]
By definition,
\begin{eqnarray*}
&& \left|\iprod{\hat{\theta}-\tilde{\theta}}{Y-\theta}\right|  \\
 &= & \left|\sum_{a\in [\dimkone], b \in [\dimktwo]} \Big( \text{sign}(\bar{Y}_{ab}(\hat{z}))\left(|\bar{Y}_{ab}(\hat{z})| \wedge M \right) - \bar{\theta}_{ab}(\hat{z}) \Big) \bar{W}_{ab}(\hat{z}) |\hat{z}_1^{-1}(a)| |\hat{z}_2^{-1}(b)|\right| \\
&\le& \max_{z\in\mathcal{Z}_{k_1k_2}} \left|\sum_{a\in [\dimkone], b \in [\dimktwo]} \Big( \text{sign}(\bar{Y}_{ab}(z))\left(|\bar{Y}_{ab}(z)| \wedge M \right) - \bar{\theta}_{ab}(z) \Big)\bar{W}_{ab}(z) |z_1^{-1}(a)| |z_2^{-1}(b)|\right|. 
\end{eqnarray*}
By definition, we have 
$$ \Big( \text{sign}(\bar{Y}_{ab}(z))\left(|\bar{Y}_{ab}(z)| \wedge M \right) - \bar{\theta}_{ab}(z) \Big)\bar{W}_{ab}(z) \le  |\bar{W}_{ab}(z)|^2 \wedge \tau |\bar{W}_{ab}(z)|. $$
For any fixed $z_1\in[k_1]^{n_1},z_2\in[k_2]^{n_2}$, define $n_1(a)=|z_1^{-1}(a)|$ for $a \in [\dimkone]$, $n_2(b)=|z_1^{-1}(b)|$ for $b \in [\dimktwo]$ and $V_{ab}(z)=\sqrt{n_1(a) n_2(b)} |\bar{W}_{ab}(z)| \mathbf{1}\{|\bar{W}_{ab}(z)|\le \tau\}$, $R_{ab}(z) = \tau n_1(a) n_2(b) |\bar{W}_{ab}(z)| \mathbf{1}\{|\bar{W}_{ab}(z)|>\tau\}$.
Then
\begin{equation*}
\left|\iprod{\hat{\theta}-\tilde{\theta}}{Y-\theta}\right| \le \max_{z\in\mathcal{Z}_{k_1k_2}} \left \{ \sum_{a\in [\dimkone], b \in [\dimktwo]} V_{ab}^2(z) + \sum_{a\in [\dimkone], b \in [\dimktwo]} R_{ab}(z) \right \}.
\end{equation*}
Following the same argument in the proof of Lemma \ref{lem:average}, a choice of $t =C_3(M^2\vee\sigma^2)(\dimkone\dimktwo+\dimnone \log \dimkone + \dimntwo \log \dimktwo)/p$ for some sufficiently large $C_3>0$ will complete the proof.
\end{proof}

\begin{proof}[Proof of Lemma \ref{lm:subexp}]
When $|\lambda| \le p/(M \vee \sigma)$, $|\lambda \theta_i/p| \le 1$ and $\lambda^2 \sigma^2/p^2 \le 1$. Then
\begin{eqnarray*}
\E e^{\lambda (Y_i-\theta_i)} &=& p \E e^{\lambda(X/p-\theta_i)} + (1- p) \E e^{-\lambda \theta_i} \\
&\le&  p e^{\frac{\lambda^2 \sigma^2}{2p^2} + \frac{1-p}{p} \lambda \theta_i} + (1-p) e^{ - \lambda \theta_i}  \\
&\leq& p\left(1+\frac{\sigma^2\lambda^2}{p^2}\right)\left(1+\frac{1-p}{p}\lambda\theta_i + \frac{2(1-p)^2}{p^2}\lambda^2\theta_i^2\right) + (1-p)(1-\lambda\theta_i+2\lambda^2\theta_i^2) \\
&=& 1 + \left[\frac{2(1-p)^2}{p}+2(1-p)\right]\lambda^2\theta_i^2+ \frac{1}{p}\lambda^2\sigma^2 + \frac{1-p}{p^2}\lambda^3\theta_i\sigma^2 + \frac{2(1-p)^2}{p^3}\lambda^4\theta_i^2\sigma^2 \\
&\leq& 1 + \frac{2}{p}\lambda^2\theta_i^2 + \frac{1}{p}\lambda^2\sigma^2 + \frac{1}{p^2}\lambda^3\theta_i\sigma^2 + \frac{2}{p^3}\lambda^4\theta_i^2\sigma^2 \\
&\leq& 1 + \frac{2}{p}\lambda^2\theta_i^2 + \frac{1}{p}\lambda^2\sigma^2 + \frac{1}{p}\lambda^2\sigma^2 + \frac{2}{p}\lambda^2\sigma^2 \\
&=& 1 + \left(2\theta_i^2+4\sigma^2\right)\frac{\lambda^2}{p} \\
&\leq& e^{2(M^2+2\sigma^2)\frac{\lambda^2}{p}}.
\end{eqnarray*}
The second inequality is due to the fact that $e^x \le 1+ 2x$ for all $x\geq 0$ and $e^x \le 1+ x + 2x^2$ for all $|x| \le 1$.
Then for $|\lambda| (M \vee \sigma)\norm{c}_\infty \le p$, Markov inequality implies
\begin{equation*}
\Prob \left( \sum_{i=1}^{n} c_{i} (Y_i - \theta_i) \ge t\right)  \le \exp \left\{ - \lambda t + \frac{2\lambda^2}{p}(M^2 + 2 \sigma^2) \right\}.
\end{equation*}
By choosing $\lambda=\min\left\{\frac{pt}{4(M^2+2\sigma^2)}, \frac{p}{(M \vee \sigma)||c||_\infty}\right\}$, we get (\ref{eq:bernstein1}).
\end{proof}

\begin{proof}[Proof of Corollary \ref{cor:SBM}]
For independent Bernoulli random variables $X_{i} \sim Ber(\theta_i)$ with $\theta_i \in [0, \rho]$ for $i \in [n]$. Let $Y_i = X_i E_i/p$, where $\{E_i\}$ are indenpendent Bernoulli random variables and $\{E_i\}$ and $\{X_i\}$ are independent. Note that $\E Y_i = p_i$, $\E Y_i^2 \le \rho/p$ and $|Y_i| \le 1/p$. Then Bernstein's inequality \cite[Corollary 2.10]{massart2007} implies
\begin{equation} \label{eq:bernstein2}
\Prob \left\{ \left| \frac{1}{\sqrt{n}} \sum_{i=1}^{n} (Y_i - \theta_i) \right| \ge t \right\} \le 2 \exp \left\{ - \min \left( \frac{pt^2}{4\rho}, \frac{3\sqrt{n} pt}{4}  \right) \right\}
\end{equation}
for any $t>0$. Let $S=\frac{1}{\sqrt{n}} \sum_{i=1}^{n} (Y_i - \theta_i)$, $T=S \mathbf{1}\{|S| \le 3\rho \sqrt{n}\}$ and $R=3 \rho \sqrt{n} \mathbf{1}\{|S| > 3\rho \sqrt{n}\}$. Following the same arguments as in the proof of Lemma \ref{lm:subexV}, we have $\E e^{pT^2/ (8\rho)} \le 5$ and $\E e^{p R/ (8 \rho)} \le 9$. Consequently, Lemma \ref{lem:average}, Lemma \ref{lem:partition1} and Lemma \ref{lem:partition2} hold for the Bernoulli case. Then the rest of the proof follows from the proof of Theorem \ref{thm:main}.
\end{proof}

\begin{proof}[Proof of Lemma \ref{lem:aaa}]
By the definitions of $Y$ and $\mathcal{Y}$, we have
$$\norm{Y-\mathcal{Y}}^2\leq (\hat{p}^{-1}-p^{-1})^2\max_{i,j}X_{ij}^2\sum_{ij}E_{ij}.$$
Therefore, it is sufficient to bound the three terms. For the first term, we have
$$|\hat{p}^{-1}-p^{-1}|\leq |\hat{p}^{-1}-p^{-1}|\frac{|\hat{p}-p|}{p}+\frac{|\hat{p}-p|}{p^2},$$
which leads to
\begin{equation}
|\hat{p}^{-1}-p^{-1}|\leq \left(1-\frac{|\hat{p}-p|}{p}\right)^{-1}\frac{|\hat{p}-p|}{p^2}.\label{eq:p-bound}
\end{equation}
Bernstein's inequality implies $|\hat{p}-p|^2\leq C\frac{p\log(n_1+n_2)}{n_1n_2}$ with probability at least $1-(n_1n_2)^{-C'}$ under the assumption that $p\gtrsim \frac{\log(n_1+n_2)}{n_1n_2}$. Plugging the bound into (\ref{eq:p-bound}), we get
$$(\hat{p}^{-1}-p^{-1})^2\leq C_1\frac{\log(n_1+n_2)}{p^3n_1n_2}.$$
The second term can be bounded by a union bound with the sub-Gaussian tail assumption of each $X_{ij}$. That is,
$$\max_{i,j}X_{ij}^2\leq C_2(M^2+\sigma^2\log(n_1+n_2)),$$
with probability at least $1-(n_1n_2)^{-C'}$. Finally, using Bernstein's inequality again, the third term is bounded as
$$\sum_{ij}E_{ij}\leq C_3n_1n_2\left(p+\sqrt{\frac{p\log(n_1+n_2)}{n_1n_2}}\right)\leq C_3'n_1n_2p,$$
with probability at least $1-(n_1n_2)^{-C'}$ under the assumption that $p\gtrsim \frac{\log(n_1+n_2)}{n_1n_2}$. Combining the three bounds, we have obtained the desired conclusion.
\end{proof}

\begin{proof}[Proof of Lemma \ref{lem:aaaa}]
For the second and the third bounds, we use
$$\left|\iprod{\frac{\tilde{\theta}-\theta}{\norm{\tilde{\theta}-\theta}}}{\mathcal{Y}-\theta}\right|\leq \left|\iprod{\frac{\tilde{\theta}-\theta}{\norm{\tilde{\theta}-\theta}}}{Y-\theta}\right|+\norm{\mathcal{Y}-Y},$$
and
$$\left|\iprod{\hat{\theta}-\tilde{\theta}}{\mathcal{Y}-\theta}\right|\leq \left|\iprod{\hat{\theta}-\tilde{\theta}}{Y-\theta}\right|+\norm{\hat{\theta}-\tilde{\theta}}\norm{Y-\mathcal{Y}},$$
followed by the original proofs of Lemma \ref{lem:partition1} and Lemma \ref{lem:partition2}. To prove the first bound, we introduce the notation $\check{\theta}_{ij}=\check{Q}_{\hat{z}_1(i)\hat{z}_2(j)}$ with $\check{Q}_{ab}=\text{sign}(\bar{Y}_{ab}(\hat{z}))\left(|\bar{Y}_{ab}(\hat{z})| \wedge M\right)$. Recall the definition of $\hat{Q}$ in Proposition \ref{prop:ls} with $Y$ replaced by $\mathcal{Y}$. Then, we have
$$\norm{\hat{\theta}-\tilde{\theta}}^2\leq 2\norm{\hat{\theta}-\check{\theta}}^2+2\norm{\check{\theta}-\tilde{\theta}}^2.$$
Since $\norm{\check{\theta}-\tilde{\theta}}$ can be bounded by the exact argument in the proof of Lemma \ref{lem:average}, it is sufficient to bound $\norm{\hat{\theta}-\check{\theta}}^2$. By Jensen inequality,
$$\norm{\hat{\theta}-\check{\theta}}^2\leq \sum_{ab}|\hat{z}^{-1}(a)||\hat{z}^{-1}(b)|(\bar{Y}_{ab}(\hat{z})-\bar{\mathcal{Y}}_{ab}(\hat{z}))^2\leq \norm{Y-\mathcal{Y}}^2.$$
Thus, the proof is complete.
\end{proof}
 
\bibliographystyle{plainnat}
\bibliography{reference}


\end{document}